\documentclass[11pt,twoside, reqno]{article}

\usepackage{amsfonts}
\usepackage{amssymb}
\usepackage{amsmath}
\usepackage{amsthm}
\usepackage{latexsym}
\usepackage{mathrsfs}
\usepackage{bbm}
\usepackage{braket}

\usepackage{txfonts}
\usepackage{enumerate}
\usepackage{hyperref}
\usepackage{graphicx}
\usepackage{authblk}

\usepackage{color}
\usepackage{pgf,tikz}
\usepackage{mathrsfs}
\usetikzlibrary{arrows}
\usepackage{mdframed}
\usepackage{braket}

\allowdisplaybreaks

\newcommand{\ed}{\mathop{}\!\mathrm{d}}

\def\r{\right}
\def\lf{\left}
\def\l{\left}

\newcommand\newdot{{\kern.8pt\cdot\kern.8pt}}

\def\ptr {/\!/}

\def\e{{\rm e}}
\def\mathpal#1{\mathop{\mathchoice{\text{\rm #1}}%
   {\text{\rm #1}}{\text{\rm #1}}%
   {\text{\rm #1}}}\nolimits}
\newcommand\Hess{\mathpal{Hess}}
\newcommand\Ric{\mathpal{Ric}}

\newcommand\E{\mathbb{E}}
\newcommand\R{\mathbb{R}}

\newcommand\II{\mathrm{II}}

\newtheorem{theorem}{Theorem}[section]
\newtheorem{lemma}[theorem]{Lemma}
\newtheorem{corollary}[theorem]{Corollary}

\theoremstyle{definition}
\newtheorem{remark}[theorem]{Remark}

\newtheorem{definition}[theorem]{Definition}

\newtheorem{example}{Example}[section]

\numberwithin{equation}{section}

\def \e{{\rm e}}

\begin{document}

 \arraycolsep=1pt

\title{\bf\Large
Gradient Estimates for Neumann Semigroups on Manifolds with Boundary  under  Unbounded  Curvature Conditions
\footnotetext{\hspace{-0.35cm}
2020 Mathematics Subject Classification.
Primary 58J65, 60J60; Secondary 35K08, 53C21.
\endgraf {\it Key words and phrases}.
Bismut formula;  Feynman--Kac semigroup; 
gradient estimates; Ricci curvature; second fundamental form; reflecting diffusion processes
\endgraf 
This work was supported by
   Natural Science Foundation in China  (Grant No. 12531007). }}

\author[1]{Li-Juan Cheng}    \author[2]{Feng-Ya Lin}

\affil[1,2]{\small School of Mathematics, Hangzhou Normal
  University,\par
  Hangzhou 311121, People's Republic of China\par
   \texttt{lijuan.cheng@hznu.edu.cn;} \par
   \texttt{2024111029010@stu.hznu.edu.cn.}
    \vspace{1em}}

\date{}
\maketitle

\vspace{-0.8cm}

\begin{center}
\begin{minipage}{13.5cm}
{\small {\bf Abstract} \quad
This paper establishes Bismut-type formulas and gradient estimates for Feynman--Kac semigroups on Riemannian manifolds with boundary, under geometric conditions
formulated in terms of  Ricci curvature $\mathrm{Ric}_Z \geq K$ and second fundamental form $\mathrm{II} \geq \sigma$ for potentially unbounded functions $K$ and $\sigma$.  We then apply these formulas to derive pointwise gradient estimates for the
Neumann semigroup under variable, possibly unbounded, lower curvature bounds.
Both convex and non-convex boundary cases are treated. In the non-convex case,
the boundary contribution is controlled by a conformal change of metric and
an exponential estimate for the boundary local time.}
\end{minipage}
\end{center}

\section{Introduction}

Let \(M\) be a \(d\)-dimensional complete Riemannian manifold, possibly with
non-empty boundary \(\partial M\), and let
\[
        L=\Delta+Z,
\]
where \(\Delta\) is the Laplace--Beltrami operator and \(Z\) is a smooth vector
field on \(M\). We consider the reflecting diffusion process \(x_t^x\)
generated by \(\frac12L\), starting from \(x\in M\), which is described by
\[
        dx_t^x
        =
        /\!/_{t}\circ dB_t
        +
        \frac12 Z(x_t^x)\,dt
        +
        \frac12 N(x_t^x)\,dl_t^x,
        \qquad
        x_0^x=x.
\]
Here \(B_t\) is a Brownian motion on \(T_xM\simeq\mathbb R^d\),
\(/\!/_{t}:T_xM\to T_{x_t^x}M\) denotes stochastic parallel transport,
\(N\) is the inward unit normal vector field on \(\partial M\), and \(l_t^x\)
is the boundary local time. Throughout the paper we assume that the reflecting
diffusion is non-explosive. The corresponding Neumann semigroup is
\[
        P_tf(x)=\mathbb E^x[f(x_t)],
        \qquad
        t\geq0,\quad f\in\mathcal B_b(M).
\]

Gradient estimates for heat semigroups are a classical topic connecting
stochastic analysis, geometric analysis and partial differential equations.
A fundamental tool is the Bismut-type derivative formula, which represents
derivatives of semigroups by stochastic integrals. Since Bismut's pioneering
work \cite{Bismut:1984}, such formulas have been developed extensively; see,
for instance,
\cite{EL94,Tha97,DTh2001,APT2003,St00,StT98}. They have important
applications to heat kernel estimates, functional inequalities, strong Feller
properties and sensitivity analysis; see, among others,
\cite{BBM2007,GW12,PZ95,WZh13}.

When the manifold has a boundary, the problem becomes more delicate. The
reflecting diffusion interacts with the geometry of \(\partial M\) through
the boundary local time, and the derivative process is influenced both by the
interior curvature and by the second fundamental form of the boundary. We write
\[
        \Ric_Z:=\Ric-\nabla Z^\flat
\]
for the Bakry--Emery curvature associated with \(L\), where \(Z^\flat\) is the one-form dual to \(Z\), namely
\[Z^\flat(v)=\langle Z,v\rangle,\qquad v\in TM.\] 
Define the second
fundamental form by
\[
        \II(X,Y)=-\langle\nabla_XN,Y\rangle,
        \qquad X,Y\in T\partial M .
\]
Bismut-type formulas and gradient estimates for reflecting diffusions and
Neumann semigroups were studied by Qian, Hsu, Wang and others; see
\cite{Qian:1997,Hsu,Wang2009SecondFundamentalForm,Wbook14}. Further uniform
gradient estimates and applications can be found in
\cite{Cheng_Thalmaier_Thompson:2018}.

A basic difficulty can already be seen from the multiplicative functional
appearing in the derivative formula. If
\[
        \Ric_Z\geq -K,
        \qquad
        \II\geq-\sigma,
\]
then the corresponding multiplicative functional \(Q_t:T_xM\to T_{x_t}M\)
typically satisfies an estimate of the form
\[
        |Q_t|
        \leq
        \exp\left(
        \frac12\int_0^tK(x_s)\,ds
        +
        \frac12\int_0^t\sigma(x_s)\,dl_s
        \right).
\]
For bounded \(K\) and \(\sigma\), this leads to standard gradient estimates.
However, when the lower curvature bound is unbounded or when the boundary is
non-convex, the time integral of \(K\) and the local-time contribution from
\(\sigma\) are no longer directly controlled. This is the main obstacle
addressed in the present paper.

Our aim is to establish a new type of Bismut-type formulas and pointwise gradient estimates
for Neumann semigroups on manifolds with boundary under variable, possibly
unbounded, curvature conditions. The first main ingredient is a
Feynman--Kac version of the Bismut formula for reflecting diffusions. In this
formula, the potential contributes an additional drift term involving
\(\ed V\), whereas the effects of the interior curvature and the boundary
geometry are encoded in the multiplicative functional. This separation is
crucial: it allows us to absorb the unbounded negative part of the curvature
into a Feynman--Kac weight, rather than estimating it directly through the
derivative process. We then choose the potential according to the curvature
lower bound and use the resulting formula to derive pointwise gradient
estimates for the Neumann semigroup. This approach is related to the
Feynman--Kac method of Da Prato and Priola \cite{DaP2018}, where gradient
estimates were obtained without monotonicity assumptions. For Bismut formulas
and estimates for Feynman--Kac semigroups on manifolds without boundary, see
\cite{Th19,Li2021}.

Let \(V:M\to\mathbb R\) be a potential with \(V^{-}\in\mathcal K(M)\)  (see the definition in \eqref{Kato-def}). The Feynman--Kac semigroup is defined by
\begin{align}\label{eqn-V}
        P_t^Vf(x)
        =
        \mathbb E^x
        \left[
        \mathbb V_t f(x_t)
        \right],
        \qquad
        \mathbb V_t
        =
        \exp\left(-\int_0^tV(x_s)\,ds\right).
\end{align}
Kato potentials and their Feynman--Kac semigroups are classical in the theory
of Schr\"odinger operators; see \cite{Simon,Gu2017,GP15}. In Section~2 we prove
a Bismut-type formula for \(P_t^V\) on manifolds with boundary. Compared
with the classical Neumann Bismut formula, this result is tailored to
unbounded Kato-type potentials and is stated in a localized form before the
global exhaustion argument is applied.  This localized formulation is useful
for treating unbounded curvature lower bounds and for passing rigorously from
the penalized multiplicative functionals to the limiting boundary-adapted
multiplicative functional. This formula is
the analytic foundation for the gradient estimates of the Neumann semigroup.

We first consider the convex boundary case. Assume that, for some \(m>d\),
\[
        \Ric_Z^{m,d}
        :=
        \Ric-\nabla Z^\flat-\frac{Z\otimes Z}{m-d}
        \geq
        -\ell(\rho_o),
\]
where \(\ell\in C^1(\mathbb R_+)\) has at most quadratic growth, and suppose
that
\[
        \Ric_Z\geq -h,
        \qquad
        \II\geq0,
\]
for some non-negative \(h\in C^1(M)\). Under the pointwise moment assumptions
\[
        \sup_{s\in[0,1]}\mathbb E^x[h^2(x_s)]\leq \theta_0(x),
        \qquad
        \sup_{s\in[0,1]}\mathbb E^x[|\nabla h|^2(x_s)]
        \leq \theta_1(x),
\]
we prove that, for every \(T>0\), \(x\in M\), and
\(f\in\mathcal B_b(M)\),
\[
        |\nabla P_Tf|(x)
        \leq
        \left[
        \frac1{\sqrt{\min\{T,1\}}}
        +
        \frac14\sqrt{\theta_1(x)}
        +
        \sqrt{\theta_0(x)}
        +
        \frac18\sqrt{\theta_0(x)\theta_1(x)}
        \right]
        \|f\|_\infty .
\]
Thus the coefficients in the estimate are pointwise in \(x\) and reflect the
growth of the lower curvature bound along the diffusion.  This is the main
difference from the estimates under constant curvature lower bounds: no global
constant lower bound for \(\Ric_Z\) is imposed, and the resulting coefficient is
expressed through the actual distribution of the reflecting diffusion starting
from \(x\).  Hence the estimate remains meaningful on noncompact manifolds
whose Bakry--Emery curvature may tend to \(-\infty\) at infinity. We also point out that the moment assumptions imposed on \(h\) and
\(\nabla h\) do not by themselves guarantee the exponential integrability
needed for the multiplicative functional $Q_t$. A simple geometric example is provided by a rotationally symmetric model
manifold with metric \(g=dr^2+\varphi(r)^2g_{\mathbb S^{d-1}}\), where
\(\varphi(r)=\exp(ar^2)\) for large \(r\). Then
\[
        \Ric(\partial_r,\partial_r)
        \sim -4(d-1)a^2r^2,
\]
so the curvature lower bound is controlled by
\(h(x)=C_a(1+\rho_o(x)^2)\). Although \(h(x_s)\) and
\(|\nabla h|(x_s)\) have finite second moments on finite time intervals, the
exponential moment of \(\int_0^t h(x_s)\,ds\) diverges when \(a\) is chosen
large enough, see details in Example \ref{no-Kato} in Appendix.

The abstract moment assumptions above can be verified explicitly under radial
curvature growth conditions. In particular, if
\[
        \ell(r)=K_0(1+r^2)^\gamma,
        \qquad K_0>0,
\]
then the subquadratic case \(\gamma\in[0,1)\) leads to exponential-type
weights, while the borderline quadratic case \(\gamma=1\) leads to
polynomial-type weights. Hence the resulting estimates cover both
subquadratic and quadratic curvature growth regimes; see Theorem~\ref{h-theorem}
and Corollary~\ref{examples}.

We then treat non-convex boundaries. Suppose
\[
        \II\geq-\sigma
\]
for a non-negative boundary function \(\sigma\). The local time term in the
multiplicative functional then gives a positive boundary contribution. To
control it, we use the conformal change method developed by Wang
\cite{Wang:2007,Wbook14}. More precisely, we introduce a positive function
\(\phi\) and consider the conformally changed metric
\[
        g'=\phi^{-2}g.
\]
If
\[
        N\log\phi\geq\sigma
        \quad\text{on }\partial M,
\]
then the boundary becomes convex under \(g'\). The key additional ingredient
is the local-time estimate
\[
\mathbb E^x
\left[
\phi^{-2}(x_t)
\exp\left(\int_0^t\sigma(x_s)\,dl_s\right)
\exp\left(-\int_0^t\psi(x_s)\,ds\right)
\right]
\leq
\phi^{-2}(x),
\]
provided
\[
        -L\log\phi+2|\nabla\log\phi|^2\leq\psi .
\]
This allows the boundary local time to be absorbed into a Feynman--Kac weight.

Combining the conformal change, the local-time estimate and the
Feynman--Kac Bismut formula, we obtain a non-convex boundary gradient estimate.
If
\[
        \Ric_Z\geq -h,
        \qquad
        \II\geq-\sigma,
\]
and if there exists \(\phi\in\mathcal D(M)\) (see Definition \ref{def-DM} below for the definition of $\mathcal D(M)$) satisfying
\[
        N\log\phi\geq\sigma,
        \qquad
        -L\log\phi+2|\nabla\log\phi|^2\leq \psi,
\]
then, under the moment assumptions
\[
        \sup_{s\in[0,1]}
        \mathbb E^x[|h+\psi|^2(x_s)]\leq\alpha_0(x),
        \qquad
        \sup_{s\in[0,1]}
        \mathbb E^x[|\nabla(h+\psi)|^2(x_s)]
        \leq\alpha_1(x),
\]
we prove
\[
        |\nabla P_Tf|(x)
        \leq
        \frac{\|\phi\|_\infty}{\phi(x)}
        \left[
        \frac1{\sqrt{\min\{T,1\}}}
        +
        \frac14\sqrt{\alpha_1(x)}
        +
        \sqrt{\alpha_0(x)}
        +
        \frac18\sqrt{\alpha_0(x)\alpha_1(x)}
        \right]
        \|f\|_\infty .
\]
Under additional geometric assumptions near the boundary, Wang's construction
\cite{Wang:2007,wang2005functional,Wbook14} gives an explicit conformal factor
\(\phi\), and hence an explicit non-convex boundary estimate.

Finally, we apply these semigroup estimates to Neumann eigenfunctions. Let
\(u\in C^2(M)\) satisfy
\[
        Lu=-\lambda u
        \quad\text{in }M^\circ,
        \qquad
        Nu|_{\partial M}=0,
        \qquad
        \lambda>0.
\]
Since \(P_t\) is generated by \(\frac12L\), one has
\[
        P_tu=\e^{-\lambda t/2}u,
\]
and therefore
\[
        |\nabla u|(x)=\e^{\lambda T/2}|\nabla P_Tu|(x).
\]
Choosing \(T=\lambda^{-1}\) gives pointwise gradient estimates for Neumann
eigenfunctions. Related gradient and Hessian estimates on compact manifolds
with boundary were obtained in
\cite{ArnaudonThalmaierWang2020,ChengThalmaierWang2024Hessian}. In our setting,
both convex and non-convex boundary cases are covered, and the leading
high-frequency term is of order
\[
        \sqrt{\lambda}\,\|u\|_\infty .
\]

The paper is organized as follows. In Section~2 we establish Bismut-type
formulas for Feynman--Kac semigroups on manifolds with boundary. In
Section~3 we study the convex boundary case and derive pointwise gradient
estimates under unbounded curvature lower bounds. In Section~4 we treat
non-convex boundaries by conformal change and a boundary local-time estimate.
Finally, Section~5 applies the semigroup estimates to Neumann eigenfunctions.

\paragraph{Notation convention.}
Throughout this paper,  we distinguish the exterior differential from ordinary
time, stochastic, and local-time differentials. The exterior differential of a
function is denoted by the upright roman symbol \(\ed\); for example,
\(\ed f\), \(\ed V\), and \(\ed P_t^V f\) are one-forms. By contrast, the
usual differentials in stochastic calculus and time integration are kept in the
standard italic notation, such as \(dB_t\), \(dt\), and \(dl_t\). The covariant
differential along the reflecting diffusion is denoted by \({\bf D}\).

\section{Bismut-Type Formula for Feynman--Kac Semigroups}

The usual Kato class $\mathcal{K}(M)$ is defined as the set of functions $f$ such that
\begin{align}\label{Kato-def}
\lim_{\alpha\downarrow  0} \sup_{x\in M}  \int_0^{\alpha} \E^x\bigl[|f(x_s)|\bigr] \, ds=0,
\end{align}
which  plays an important role in the study of Schr\"{o}dinger
operators and their semigroups, see Simon  \cite{Simon}  and the references therein. Under standard local heat-kernel estimates one has useful $L^p$-criteria for membership in the Kato class; see the cited references  \cite{Gu2017,GP15}.
By \cite[Lemma 3.9]{GP15}, it was proved that if $f\in \mathcal{K}(M)$, there exist constants
$c, \, C$ such that  for each $t>0$, 
\begin{align*}
\sup_{x\in M}\E^x\Bigl[ \e^{\int_0^t |f|(x_s)\, ds}\Bigr] \leq C\e^{ct}.
\end{align*}
To study the Neumann Feynman--Kac semigroup, we introduce the following  function space. 
Let
\begin{align*}
\mathcal{C}_{N}(L_V) &= \left\{ f \in C^2(M)\cap \mathcal{B}_b(M):\ Nf|_{\partial M} = 0,\ \frac{1}{2}Lf-Vf \in \mathcal{B}_b(M) \right\}.
\end{align*}
We first introduce some properties for the Neumann Feynman--Kac  semigroup $P_{\cdot}^V$:
\begin{theorem}\label{FK-prop}
Let \( f \in \mathcal{C}_N(L_V) \),  and $V^-\in \mathcal{K}(M)$ and $V\in  C^1(M)$. Then, 
\begin{enumerate}
    \item[(1)] \( \frac{d}{dt}P_t^V f = P^V_t (\frac{1}{2}L-V) f = (\frac{1}{2}L-V) P^V_t f, \quad t > 0; \)
    \item[(2)] \( N P_t^V f|_{\partial M} = 0, \quad t > 0. \)
\end{enumerate}
\end{theorem}

\begin{proof}
We first justify the regularity of the Feynman--Kac semigroup used below.
Since \(V^-\in \mathcal K(M)\), the Khasminskii-type estimate for Kato
potentials implies that, for every \(T>0\),
\[
 \sup_{x\in M}\mathbb E^x\exp\left(\int_0^T V^-(x_s)\,ds\right)<\infty .
\]
Consequently the Feynman--Kac semigroup
\[
 P_t^V f(x)
 =
 \mathbb E^x\left[
 \exp\left(-\int_0^t V(x_s)\,ds\right)f(x_t)
 \right]
\]
is well defined on bounded measurable functions. Moreover, by the Markov
property of the reflecting diffusion, \(P_t^V\) satisfies the Duhamel formula
\[
 P_t^V f
 =
 P_t f
 -
 \int_0^t P_{t-s}\big(VP_s^Vf\big)\,ds .
\]
Thus \(u(t,x):=P_t^Vf(x)\) is a mild solution of the Neumann parabolic
problem
\[
\begin{cases}
 \partial_t u=\left(\frac12L-V\right)u, & x\in M^\circ,\ t>0,\\
 Nu=0, & x\in \partial M,\ t>0,\\
 u(0,\cdot)=f.
\end{cases}
\]
The Kato condition and the associated Feynman--Kac semigroups are classical;
see, for instance, Simon~\cite{Simon} and G{\"u}neysu~\cite{Gu2017}.
Since \(V\in C^1(M)\) and the boundary is smooth, standard parabolic
regularity for Neumann boundary value problems upgrades this mild solution
to a strong solution on \((0,T]\times M\); see
Ouhabaz~\cite{Ouhabaz2005}, Ladyzhenskaya--Solonnikov--Ural'tseva~\cite{LSU1968},
and Lunardi~\cite{Lunardi1995}. 
In particular, for
\(f\in \mathcal{C}_N(L_V)\), one has
\[
u \in C^{1,2}((0,T]\times M).
\]

It remains to prove the first equality in (1). This follows from the identity \( P_t^V f = f + \int_0^t P_s^V \big(\frac{1}{2}L-V\big) f \, ds \), which is implied by Itô's formula applied to
\begin{align*}
d \Big[\e^{-\int_0^tV(x_s)\, ds}f(x_t)\Big] &\overset{m}{=} \e^{-\int_0^tV(x_s)\, ds}\Big(-V(x_t)f(x_t)+\frac{1}{2}Lf(x_t)\Big)\, dt.
\end{align*}

\end{proof}

To state the main result, we first recall the construction of the multiplicative functional
$Q_{\bullet}$ appearing in the Bismut formula; see 
{Hsu}~\cite{Hsu} and
\cite[Theorem~1.1]{Wbook14} for the case of manifolds with boundary.

For $t\geq0$, let $/\!/_{0\to t}:T_xM\to T_{x_t}M$ denote stochastic
parallel transport along the paths of the reflecting diffusion process
$\{x_{\cdot}\}$. The covariant differential ${\bf D}$ is defined by
\[
        {\bf D}:=/\!/_{0\to t}\,d\,/\!/_{t\to0},
\]
where $d$ is the usual It\^o stochastic differential in the time variable.
Let $D\Subset M$ be a relatively compact regular domain with $x\in D$, and
let
\[
        \tau:=\inf\{t\geq0:x_t\notin D\}
\]
be the first exit time of $x_t$ from $D$. By the curvature assumptions below,
there exist non-negative constants $K_0,\sigma_0$ such that
$K\geq -K_0$ on $D$ and $\sigma\geq-\sigma_0$ on $\partial M\cap D$.

For $n\in\mathbb N$, let
$Q_{t\wedge\tau}^{(n)}:T_xM\to T_{x_{t\wedge\tau}}M$ solve, up to the
stopping time $\tau$, the covariant differential equation
\begin{equation}\label{Qn-eq}
        {\bf D}Q_t^{(n)}
        =-\frac12\left\{
        \Ric_Z^\sharp(Q_t^{(n)})\,dt
        +\II^\sharp(Q_t^{(n)})\,dl_t
        +nP_N(Q_t^{(n)})\,dl_t
        \right\},
        \qquad
        Q_0^{(n)}=\mathrm{id}.
\end{equation}
Here $\mathrm{id}$ is the identity map on $T_xM$ and for $x_t\in\partial M$,
$P_N(Q_t^{(n)})$ denotes the projection onto the normal
direction, namely
\[
        P_N(Q_t^{(n)}v)
        =\langle Q_t^{(n)}v,N(x_t)\rangle N(x_t),
        \qquad v\in T_xM.
\]
Moreover, for $x_t\in\partial M$, let
$P_\partial:T_{x_t}M\to T_{x_t}\partial M$ be the tangential projection. The
operator $\II^\sharp$ is defined by
\[
        \langle \II^\sharp u,w\rangle
        =\II(P_\partial u,P_\partial w),
        \qquad u,w\in T_{x_t}M,
\]
and $\II^\sharp(Q_t^{(n)})$ in \eqref{Qn-eq} means
$\II^\sharp(Q_t^{(n)}v)$ for $v\in T_xM$.

For every fixed \(T>0\), the preceding construction gives
\begin{equation}\label{Qn-bound}
        \sup_{n\geq1,\,s\in[0,T]}
        |Q_{s\wedge\tau}^{(n)}|
        \leq
        \exp\left(
        \frac12K_0(T\wedge\tau)
        +
        \frac12\sigma_0l_{T\wedge\tau}
        \right),
\end{equation}
and
\begin{equation}\label{normal-vanish}
        \int_0^{T\wedge\tau}
        |P_N(Q_s^{(n)})|^2\,dl_s
        \leq
        \frac1n
        \int_0^{T\wedge\tau}
        |Q_s^{(n)}|^2
        \{K_0\,ds+\sigma_0\,dl_s\}
        \longrightarrow0 .
\end{equation}
as $n\to\infty$. On the other hand, by \cite[Lemma~3.1.2]{Wbook14},
\[
        \mathbb E^x\left[\exp\left(\lambda l_{T\wedge\tau}\right)\right]<\infty,
        \qquad \lambda>0.
\]
Therefore the right-hand side of \eqref{Qn-bound} is integrable on compact
time intervals.

By the proof of \cite[Theorem~1.1]{Wbook14}, there exists a progressively
measurable multiplicative functional \(Q_{\cdot\wedge\tau}\) such that, after
passing to a subsequence \(n_k\), the sequence \(Q^{(n_k)}\) converges to \(Q\)
in Hsu's weak sense. More precisely, after identifying tangent spaces along
\(x_{\cdot}\) by stochastic parallel transport,  for all bounded adapted \(T_xM\)-valued processes \(\xi_s,\eta_s\) and every bounded \(T_xM\)-valued random variables
\(\xi,\eta\), one has
\[
\lim_{k\to\infty}
\mathbb E^x
\int_0^{T\wedge\tau}
\left\langle
\bigl(Q_s^{(n_k)}-Q_s\bigr)\xi_s,\eta_s
\right\rangle\,ds
=0,
\]
and, for every fixed \(t\in[0,T]\),
\[
\lim_{k\to\infty}
\mathbb E^x
\left\langle
\bigl(Q_{t\wedge\tau}^{(n_k)}-Q_{t\wedge\tau}\bigr)\xi,\eta
\right\rangle
=0 .
\]
This weak convergence is sufficient for passing to the limit in the
stochastic-integral identities below, after localization. The uniform estimate
\eqref{Qn-bound}, together with the It\^o isometry, provides the required
square-integrability and uniform bounds for the martingale terms. Moreover, the
penalization term \(nP_N(Q_t^{(n)}(v))\,dl_t\) forces the normal component to
vanish in the limit:
\[
        \langle Q_t(v),N(x_t)\rangle=0,
        \qquad dl_t\text{-a.e. on }[0,\tau],
        \qquad v\in T_xM.
\]

\begin{theorem}\label{BisF-Neumann}
Let $D\Subset M$ be a relatively compact regular domain with $x\in D$, and
let $\tau$ be the first exit time of $x_t$ from $D$. 
Assume that there exist $K\in C(D)$ and
$\sigma\in C(\partial M\cap D)$ such that
\[
        \Ric_Z\geq K,
        \qquad
        \II\geq \sigma
        \quad \hbox{on \  }D .
\]
Then there exists a progressively measurable multiplicative functional
\(
        Q_{t\wedge\tau}:T_xM\to T_{x_{t\wedge\tau}}M
\)
such that
\[
        \langle Q_t(v),N(x_t)\rangle=0,
        \qquad dl_t\hbox{-a.e. on \  }[0,\tau],
\]
and
\[
        |Q_{t\wedge\tau}|
        \leq
        \exp\left(
        -\frac12\int_0^{t\wedge\tau}K(x_s)\,ds
        -\frac12\int_0^{t\wedge\tau}\sigma(x_s)\,dl_s
        \right),
\]
and for  $V\in C^1(M)$ and $V^-\in \mathcal K(M)$,  for every
$f\in\mathcal B_b(M)$, $v\in T_xM$, and every non-negative adapted process
$k\in L^{1,2}([0,T];\mathbb P)$ satisfying
\(
        k_0=1, \ 
        k_s=0\ \ \hbox{for }s\geq T\wedge\tau,
\)
one has
\begin{equation}\label{Bismut-F1}
\begin{aligned}
       \langle \nabla P_T^Vf, \, v\rangle
        =&-\mathbb E^x\left[
        \mathbb V_T f(x_T)
        \int_0^{T\wedge\tau}
        \left\langle Q_t(\dot k_t v),/\!/_t\,dB_t\right\rangle
        \right]
      -\mathbb E^x\left[
        \mathbb V_T f(x_T)
        \int_0^{T\wedge\tau}
        \ed V(k_tQ_t(v))\,dt
        \right].
\end{aligned}
\end{equation}
\end{theorem}

\begin{proof}
We first prove the formula for \(f\in\mathcal C_N(L_V)\) with sufficiently
bounded derivatives. The general case \(f\in\mathcal B_b(M)\) will be obtained
by a regularization argument. All the following estimates are first made
up to the localized time \(T\wedge\tau\).  On the relatively compact domain
\(D\), the functions \(V\) and \(\ed V\) are bounded from above in absolute
value, while the negative part of \(V\) is controlled globally by the Kato
assumption.  Thus the localization separates the purely local regularity
needed for differentiating \(V\) from the global integrability needed for the
Feynman--Kac weight.
Indeed, for \(\varepsilon\in(0,T)\), we apply the formula with time
\(T-\varepsilon\) to the regularized function \(P_\varepsilon^Vf\). By the
semigroup property,
\[
        P_T^Vf=P_{T-\varepsilon}^V(P_\varepsilon^Vf).
\]
The function \(P_\varepsilon^Vf\) has the required local regularity on
\((0,T]\times D\). Moreover, since $V^-\in \mathcal{K}(M)$, the
Khasminskii estimate for Kato potentials give
\[
        \sup_{x\in D}
        \mathbb E^x
        \exp\left(
        \int_0^T V(x_s)^-\,ds
        \right)<\infty.
\]

Consequently, for bounded \(f\),
\[
        |\mathbb V_T f(x_T)|
        \le \|f\|_\infty
        \exp\left(\int_0^T V^-(x_s)\,ds\right),
\]
and the right-hand side is uniformly integrable on compact time intervals.
Together with the localized bound for \(Q^{(n)}\) and the boundedness of
\(\ed V\) on \(D\), this gives an integrable dominating function for the
stochastic and drift terms in the formula.
 Hence the right-hand side converges as
\(\varepsilon\downarrow0\) by dominated convergence and the It\^o isometry,
while the left-hand side is always \(\ed P_T^Vf(v)\). Therefore it suffices to prove
the formula for \(f\in\mathcal C_N(L_V)\).

Set
\[
        N_t^{(n)}(v):=\langle \nabla P_{T-t}^Vf, \, Q_t^{(n)}v\rangle.
\]
By It\^o's formula and the identities
\[
        \ed Lf=(\operatorname{tr}\nabla^2+\nabla_Z)\ed f-\ed f(\Ric_Z^\sharp),
        \qquad
        \ed(Vf)=f\,\ed V+V\,\ed f,
\]
where the first identity is the Weitzenb\"ock formula, we get, modulo the
differential of a local martingale,
\begin{align*}
        dN_t^{(n)}(v)
        \stackrel{m}{=}&
        \ed P_{T-t}^Vf({\bf D}Q_t^{(n)}v)
        +\partial_t(\ed P_{T-t}^Vf)(Q_t^{(n)}v)\,dt
        \\
        &+\frac12(\operatorname{tr}\nabla^2+\nabla_Z)(\ed P_{T-t}^Vf)(Q_t^{(n)}v)\,dt
        +\frac12\nabla_N(\ed P_{T-t}^Vf)(Q_t^{(n)}v)\,dl_t
        \\
        =&\frac12\nabla_N(\ed P_{T-t}^Vf)(Q_t^{(n)}v)\,dl_t
        -\frac12\ed P_{T-t}^Vf(\II^\sharp(Q_t^{(n)}v))\,dl_t
        \\
        &-\frac n2\ed P_{T-t}^Vf(P_N(Q_t^{(n)}v))\,dl_t
        +V(x_t)N_t^{(n)}(v)\,dt
        +P_{T-t}^Vf(x_t)\ed V(Q_t^{(n)}v)\,dt.
\end{align*}
Since $NP_{T-t}^Vf=0$ on $\partial M$, the boundary terms combine as
\[
        \nabla_N(\ed P_{T-t}^Vf)(Q_t^{(n)}v)
        -\ed P_{T-t}^Vf(\II^\sharp(Q_t^{(n)}v))
        =
        \Hess_{P_{T-t}^Vf}(N,N)\langle Q_t^{(n)}v,N\rangle
\]
on $\partial M$. Recalling the definition of $\mathbb V_t$ in \eqref{eqn-V},
we obtain
\begin{equation}\label{Ito-formula1}
        d(\mathbb V_tN_t^{(n)}(v))
        \stackrel{m}{=}
        \frac12\mathbb V_t\Hess_{P_{T-t}^Vf}(N,N)
        \langle Q_t^{(n)}v,N\rangle\,dl_t
        +\mathbb V_tP_{T-t}^Vf(x_t)\ed V(Q_t^{(n)}v)\,dt.
\end{equation}

Let $k\in L^{1,2}([0,T];\mathbb P)$. Applying \eqref{Ito-formula1} to
$k_tv$ gives that
\begin{align*}
        &\mathbb V_tN_t^{(n)}(k_tv)
        -\int_0^t\mathbb V_s \ed P_{T-s}^Vf(Q_s^{(n)}(\dot k_sv))\,ds
        -\int_0^t\mathbb V_sP_{T-s}^Vf(x_s)\ed V(Q_s^{(n)}(k_sv))\,ds
        \\
        &\quad
        -\frac12\int_0^t\mathbb V_s\Hess_{P_{T-s}^Vf}(N,N)
        \langle Q_s^{(n)}(k_sv),N\rangle\,dl_s
\end{align*}
is a local martingale. Moreover, using the equation solved by
$P_{T-t}^Vf$, we have
\[
        d P_{T-t}^Vf(x_t)
        =V(x_t)P_{T-t}^Vf(x_t)\,dt
        +\langle \nabla P_{T-t}^Vf(x_t),/\!/_t\,dB_t\rangle.
\]
Therefore
\[
        \mathbb V_tP_{T-t}^Vf(x_t)
        =P_T^Vf(x)
        +\int_0^t\mathbb V_s
        \langle \nabla P_{T-s}^Vf(x_s),/\!/_s\,dB_s\rangle.
\]
It follows that
\[
        \int_0^t\mathbb V_s \ed P_{T-s}^Vf(Q_s^{(n)}(\dot k_sv))\,ds
        -\mathbb V_tP_{T-t}^Vf(x_t)
        \int_0^t\langle Q_s^{(n)}(\dot k_sv),/\!/_s\,dB_s\rangle
\]
is a local martingale. Hence
\begin{align}\label{martingale-1}
        &\mathbb V_tN_t^{(n)}(k_tv)
        -\mathbb V_tP_{T-t}^Vf(x_t)
        \int_0^t\langle Q_s^{(n)}(\dot k_sv),/\!/_s\,dB_s\rangle
        \notag\\
        &\quad
        -\int_0^t\mathbb V_sP_{T-s}^Vf(x_s)\ed V(Q_s^{(n)}(k_sv))\,ds
        \notag\\
        &\quad
        -\frac12\int_0^t\mathbb V_s\Hess_{P_{T-s}^Vf}(N,N)
        \langle Q_s^{(n)}(k_sv),N\rangle\,dl_s
\end{align}
is a local martingale.

Taking $t=T\wedge\tau$ in \eqref{martingale-1}, and using $k(0)=1$ and
$k_s=0$ for $s\geq T\wedge\tau$, we obtain
\begin{align}\label{formula2}
        -\ed P_T^Vf(v)
        =&\mathbb E^x\left[
        \mathbb V_{T\wedge\tau}P_{T-T\wedge\tau}^Vf(x_{T\wedge\tau})
        \int_0^{T\wedge\tau}
        \langle Q_s^{(n)}(\dot k_sv),/\!/_s\,dB_s\rangle
        \right]
        \notag\\
        &+\frac12\mathbb E^x\left[
        \int_0^{T\wedge\tau}\mathbb V_tk_t
        \Hess_{P_{T-t}^Vf}(N,N)
        \langle Q_t^{(n)}v,N\rangle\,dl_t
        \right]
        \notag\\
        &+\mathbb E^x\left[
        \int_0^{T\wedge\tau}\mathbb V_tP_{T-t}^Vf(x_t)
        \ed V(k_tQ_t^{(n)}v)\,dt
        \right].
\end{align}
By the strong Markov property of $X_t$ and the semigroup property of $P_t^V$,
\begin{align*}
        &\mathbb E^x\left[
        \mathbb V_{T\wedge\tau}P_{T-T\wedge\tau}^Vf(x_{T\wedge\tau})
        \int_0^{T\wedge\tau}
        \langle Q_s^{(n)}(\dot k_sv),/\!/_s\,dB_s\rangle
        \right]
        \\
        &\qquad=
        \mathbb E^x\left[
        \mathbb V_Tf(x_T)
        \int_0^{T\wedge\tau}
        \langle Q_s^{(n)}(\dot k_sv),/\!/_s\,dB_s\rangle
        \right],
\end{align*}
and similarly
\begin{align*}
        \mathbb E^x\left[
        \int_0^{T\wedge\tau}\mathbb V_tP_{T-t}^Vf(x_t)
        \ed V(k_tQ_t^{(n)}v)\,dt
        \right]
        =
        \mathbb E^x\left[
        \mathbb V_Tf(x_T)
        \int_0^{T\wedge\tau}\ed V(k_tQ_t^{(n)}v)\,dt
        \right].
\end{align*}
Substituting these identities into \eqref{formula2} yields
\begin{align}\label{formula3}
        -\ed P_T^Vf(v)
        =&\mathbb E^x\left[
        \mathbb V_Tf(x_T)
        \int_0^{T\wedge\tau}
        \langle Q_s^{(n)}(\dot k_sv),/\!/_s\,dB_s\rangle
        \right]
        \notag\\
        &+\frac12\mathbb E^x\left[
        \int_0^{T\wedge\tau}\mathbb V_tk_t
        \Hess_{P_{T-t}^Vf}(N,N)
        \langle Q_t^{(n)}v,N\rangle\,dl_t
        \right]
        \notag\\
        &+\mathbb E^x\left[
        \mathbb V_Tf(x_T)
        \int_0^{T\wedge\tau}\ed V(k_tQ_t^{(n)}v)\,dt
        \right].
\end{align}

On the localized interval $[0,T\wedge\tau]$, the Hessian term is bounded for
$f\in\mathcal C_N(L_V)$ and $t<T$. Moreover, by \eqref{normal-vanish},
\[
        \int_0^{T\wedge\tau}|P_N(Q_s^{(n)})|^2\,dl_s\to0.
\]
Therefore the boundary Hessian term in \eqref{formula3} converges to zero. 

Indeed, on \([0,T\wedge\tau]\) the Hessian factor is bounded after the
initial regularization, and the normal factor satisfies
\(\langle Q_t^{(n)}v,N(x_t)\rangle=P_N(Q_t^{(n)}v)\).  Hence Cauchy's
inequality with respect to \(dl_t\) and \eqref{normal-vanish} give the desired
vanishing.
Passing to the subsequence $n_k$ for which the weak convergence of
$Q^{(n_k)}$ holds, and then letting $k\to\infty$, we get
\[
\begin{aligned}
        \ed P_T^Vf(v)
        =&-\mathbb E^x\left[
        \mathbb V_Tf(x_T)
        \int_0^{T\wedge\tau}
        \left\langle Q_t(\dot k_tv),/\!/_t\,dB_t\right\rangle
        \right]
        \\
        &-\mathbb E^x\left[
        \mathbb V_Tf(x_T)
        \int_0^{T\wedge\tau}\ed V(k_tQ_t(v))\,dt
        \right].
\end{aligned}
\]
This proves the formula for  $f\in\mathcal C_N(L_V)$. \end{proof}

\section{Gradient Estimates for Neumann Semigroups with Convex Boundary}

\subsection{The Main Estimate}
For fixed $o\in M$, we write $\rho_o(p):=\rho(o, p)$ for simplicity.
We begin  our discussion  with the following curvature condition:

\begin{mdframed}  Condition {\bf (C1)}:  There exists a non-negative function $\ell \in C^1(\R^+)$ with $\ell(r)\leq c_0(1+r^2)$ for some non-negative constant $c_0$, and  $m>d$  such that
\begin{align}\label{CD-condition}
	\Ric_Z^{m,d}:=\Ric-\nabla Z^\flat-\frac{Z\otimes Z}{m-d}\geq -\ell (\rho_o(\cdot)),  \tag{{\bf CD}$(m, \ell(\rho_o))$}
\end{align}
and  there exists a non-negative function $h\in C^1(M)$ such that 
\begin{align}\label{c1}
\Ric(x)-\nabla Z^\flat(x)\geq -h(x),\quad \II\geq 0.
\end{align}
\end{mdframed}
 
Note that the curvature-dimension condition \eqref{CD-condition} implies \eqref{CD-ineq}, which  ensures the non-explosion of the
process $x_{\cdot}$ (see \cite[Theorem 2.1.1]{Wbook14}), i.e. $\eta_x=\infty$ and it also implies the strong 1-completeness of the process, which further implies that  the global version of the Bismut formula in \eqref{Bismut-F1} holds, see Lemma \ref{lem} below.

Recall $L=\Delta+Z$ and  define  the generator $L^h:=\frac{1}{2}(L-h), \ h \in C^1(M)$, and let $P_t^h$ be the associated Feynman--Kac semigroup, i.e. 
\begin{align*}
P_t^hf(x)=\E^x\lf[\mathbb{H}_t f(x_t)\r],
\end{align*}
where 
\begin{equation}
    \mathbb{H}_t:=\e^{-\frac{1}{2}\int_0^th(x_s)\, ds}.
\end{equation}
\begin{lemma}\label{lem}
   Fix  $x\in M$ and $v\in T_xM$, let $x_{\cdot}^x$ be a $L$-diffusion starting at $x$.  Under the curvature condition $(\bf C1)$ and 
   $$  \mathbb{E}[|\nabla h|^2(x_t^x)]<\infty,$$
  for  $k\in C^1([0,\, T])$ such that $ k(0)=1, k(T)=0$, we have that for $f\in \mathcal{B}_b(M)$,
\begin{equation}\label{s1.fy.bis.loc}
\aligned
    \Braket{\nabla P_T^hf, v}&=-\mathbb{E}\lf[\mathbb{H}_Tf(x_T)\int_0^T\Braket{Q_s(\dot{k}(s)v), \ptr_sdB_s}\r]\\
    &\quad -\frac{1}{2}\mathbb{E}\lf[\mathbb{H}_Tf(x_T)\int_0^T\Braket{\nabla h , Q_s(k(s)v)}\, ds\r],
\endaligned
\end{equation}
where \( Q_t : T_xM \to T_{x_t^x}M \) is the same as in Theorem \ref{BisF-Neumann}.
\end{lemma}
\begin{proof}
 Let $(D_n)_{n\ge1}$ be a smooth relatively compact exhaustion of $M$ such that $x\in D_1$, $\overline D_n\subset D_{n+1}$ and $\bigcup_nD_n=M$. Denote by $\tau_n$ the first exit time of $x_t$ from $D_n$.  
For $v\in T_{x} M$ and $k_n\in L^{1,2}([0,T];\mathbb{P})$  a non-negative adapted process  such that $k_n(0)=1$, $k_n(s)=0$ for $s\geq T\wedge \tau_{n}$,
\begin{align}\label{Bismut-F}
    \ed P_T^h f(v)= -\E\l[ \mathbb{H}_T f(x_T)\int_0^{T} \langle Q_t(\dot{k}_n(t)v), /\!/ _t \,dB_t\rangle\r]
    -\frac{1}{2}\E\l[\mathbb{H}_T f(x_T)\int_0^{T} k_n(t) \, \ed h (Q_t(v))\, dt \r]. 
\end{align} 
We first see that  \(|\mathbb{H}_T Q_t |\leq 1,\ t\leq T\) 
and
\begin{align*}
    |\ed P_T^h f|&\leq \Big|\E\l[ \mathbb{H}_T f(x_T)\int_0^{T} \langle Q_t(\dot{k}_n(t)v), /\!/ _t \,dB_t\rangle\r]\Big|+
   \frac{1}{2} \Big|\E\l[\mathbb{H}_T f(x_T)\int_0^{T} k_n(t) \, \ed h (Q_t(v))\, dt \r]\Big|\\
   &\leq \|f\|_{\infty} \E\l[\int_0^{T}|\dot{k}_n(t)|^2 \,d t\r]^{1/2}+\frac{1}{2}\|f\|_{\infty} T \sup_{s\in [0,T]}\E |\nabla h|^2(x_s).
\end{align*} 
As discussed in \cite{TW98}, there exists $\beta_n \in C^2(D_n)$ such that $\beta_n(x)=1$, $\beta_n|_{\partial D_n}=0$ and $c(\beta_n):=\sup_{D_n}\{- \beta_nL\beta_n+3 |\nabla \beta_n|^2\}<\infty$,
\begin{align}\label{beta-n}
\E\l[\int_0^{T}|\dot{k}_n(t)|^2 \,d t\r]\leq \frac{c(\beta_n)}{1-\e^{-c(\beta_n)T}}.
\end{align}
By Lemma \ref{lem3.3} below, it is easy to see that  under the curvature condition {\bf (C1)}, when $n$ tends to $\infty$, there
exists a constant $c>0$ such that
\begin{align*}
\limsup_{n\rightarrow \infty}\E\l[\int_0^{T}|\dot{k}_n(t)|^2 \,d t\r]\leq \frac{1}{T}+c.
\end{align*}
 Then  we  know that
$|\nabla P_{\cdot}^hf|$ is bounded on $[\epsilon, T] \times M$ for any small $\epsilon>0$.  For small $\epsilon\in (0,\,T)$, let $k$ be a function
in $C^1([0,T-\epsilon])$. 
Note that 
\begin{align*}
   &\mathbb{H}_t N_t^{(n)}(k_t v) - \mathbb{H}_t P_{T-t}^hf(x_t) \int_0^t \langle Q_s^{(n)}(\dot{k}_s v), /\!/_s \,dB_s \rangle  - \frac{1}{2}\int_0^t \mathbb{H}_s P_{T-s}^h f(x_s) \ed h(Q_s^{(n)}(k_s v)) \,ds 
   \\
   &-\frac{1}{2} \int_0^t \mathbb{H}_s \Hess_{P_{T-s}^hf}(N,N)\langle Q_s^{(n)}(k_sv),\, N \rangle \, dl_s
\end{align*}
is a local martingale, where $N_t^{(n)}(v)=\ed P_{T-t}^h f(Q_t^{(n)}(v))$ for $v\in T_xM$.  Then
\begin{align*}
\ed P_t^{h}f(v)=&\E\l[k_{{(T-\epsilon)}\wedge \tau_m}\mathbb{H}_{(T-\epsilon)\wedge \tau_m} \ed P_{T-(T-\epsilon)\wedge \tau_m}^{h}f(Q^{n}_{(T-\epsilon)\wedge \tau_m}(v))  \r]\\
&-\E \l[\mathbb{H}_{T} f(x_T) \int_0^{{(T-\epsilon)}\wedge \tau_m} \langle Q_s^{(n)}(\dot{k}_s v), /\!/_s \,dB_s \rangle\r]\\
&-\frac{1}{2}\E\l[\mathbb{H}_T f(x_T) \int_0^{{(T-\epsilon)}\wedge \tau_m } \ed h(Q_s^{(n)}(k_s v))\, ds \r]\\
&-\frac{1}{2} \E \l[\int_0^{{(T-\epsilon)}\wedge \tau_m } \mathbb{H}_s \Hess_{P_{T-s}^h f}(N,N)\langle Q_s^{(n)}(k_sv),\, N \rangle \, dl_s\r].
\end{align*}
By the condition $\E h^2(x_s)< \infty$ and $\E |\nabla h|^2(x_s)< \infty$,  we  first let $m$ tend to $\infty$ and then let $n$ tend to $\infty$ to  obtain
\begin{align*}
\ed P_t^{h}f(v)&=\lim_{m\rightarrow \infty}\E\l[k_{{(T-\epsilon)}\wedge \tau_m}\mathbb{H}_{(T-\epsilon)\wedge \tau_m} \ed P_{T-(T-\epsilon)\wedge \tau_m}^{h}f(Q_{(T-\epsilon)\wedge \tau_m}(v))  \r]\\
&\quad-\lim_{m\rightarrow \infty}\E \l[\mathbb{H}_{T} f(x_T) \int_0^{{(T-\epsilon)}\wedge \tau_m} \langle Q_s(\dot{k}_s v), /\!/_s \,dB_s \rangle\r]\\
&\quad-\frac{1}{2}\lim_{m\rightarrow \infty}\E\l[\mathbb{H}_T f(x_T) \int_0^{{(T-\epsilon)}\wedge \tau_m } \, \ed h(Q_s(k_s v))\, ds \r]\\
&=-\E \l[\mathbb{H}_{T} f(x_T) \int_0^{{(T-\epsilon)}} \langle Q_s(\dot{k}_s v), /\!/_s \,dB_s \rangle\r]-\frac{1}{2}\E\l[\mathbb{H}_T f(x_T) \int_0^{{(T-\epsilon)} } \ed h(Q_s(k_s v))\, ds \r].
\end{align*}
We then get the 
global Bismut formula for $\nabla P_{\cdot}^hf $ for $ f\in \mathcal{B}_b(M)$ by letting $\epsilon$ tend to $0$. 

\end{proof}

We use the   curvature-dimension CD$(m,\ell(\rho_x(\cdot)))$ in the condition (\textbf{C1}) to  obtain an estimate of $c(\beta_n)$ in the inequality \eqref{beta-n} as follows.
\begin{lemma} \label{lem3.3}
Suppose  the curvature-dimension condition
\(CD(m,\ell(\rho_o))\) holds for some non‑negative function \( \ell \in C^1(\mathbb{R}^+) \) with \( \ell(r)\leq c_0(1+r^2) \) for some non‑negative constant $c_0$. Then for every relatively compact regular domain
\(D\subset M\) containing \(x\), there exists a function
\(\beta\in C^2(D)\) such that \(0\leq \beta\leq1\),
\(\beta(x)=1\), \(\beta|_{\partial D}=0\), and
\[
        c(\beta):=\sup_D\{-\beta L\beta+3|\nabla\beta|^2\}<\infty.
\]
Moreover, if \(D\subset B(x,\delta_x)\), then
\[
        c(\beta)
        \leq
        \left(\frac{\pi}{2\delta_x}\right)^2
        \left\{
        \frac{m}{2}+ \frac{1}{2}
        \sqrt{
        m^2+
        4mc_0\bigl(1+(\rho_o(x)+\delta_x)^2\bigr)\delta_x^2}
        +3
        \right\}
\]
for some constant \(c_0>0\).
\end{lemma}
\begin{proof}
We first argue formally with
\[
        \beta(p)=\cos\left(\frac{\pi \rho(x,p)}{2\delta_x}\right).
\]
The function \(r\) is smooth outside \(\{x\}\cup\operatorname{Cut}(x)\), and on this
set we have
\[
        |\nabla\beta|\leq \frac{\pi}{2\delta_x}.
\]
Moreover, by the Laplacian comparison theorem for the operator \(L=\Delta+Z\)
under the curvature-dimension condition
\[
        \operatorname{Ric}^{m,d}_Z\geq -\ell(\rho_o),
\]
one has, outside the cut locus,
\begin{align}\label{CD-ineq}
        L\rho^2(x,\cdot)(p)
        \leq
        m\left(
        1+\sqrt{1+\frac{4\ell(\rho_o(p))\rho^2(x,p)}{m}}
        \right);
\end{align}
see \cite[Theorem~1]{Qian}. Therefore,
\[
\begin{aligned}
        -\beta L\beta
        &\leq
        \frac{\pi}{2\delta_x \, \rho(x,p) }
        \sin\left(\frac{\pi \rho(x,p) }{2\delta_x}\right)
        \left(
        \frac12 L\rho^2(x,\cdot)(p)-1
        \right)
        +
        \left(\frac{\pi}{2\delta_x}\right)^2        \\
        &\leq
        \frac{m}{2}
        \left(\frac{\pi}{2\delta_x}\right)^2
        \left(
        1+
        \sqrt{1+\frac{4\ell(\rho_o(p))\rho^2(x,p)}{m}}
        \right).
\end{aligned}
\]
The above computation is initially carried out outside
\(\{x\}\cup\operatorname{Cut}(x)\). The cut-locus difficulty is handled in the
standard way: by Calabi's trick, the Laplacian comparison inequality extends
to all of \(D\) in the barrier sense; see \cite[Section~2.2]{Wbook14} and
\cite[Chapter~2]{Petersen2016}. Hence the estimate for \(-\beta L\beta\) also
holds in the barrier sense. By smoothing distance-type cutoff functions with
barrier inequalities, see \cite{GreeneWu1979} or
\cite[Section~2.2]{Wbook14}, we may choose a genuine \(C^2\) cutoff, still
denoted by \(\beta\), satisfying \(\beta(x)=1\), \(\beta|_{\partial D}=0\), and
\[
        c(\beta):=\sup_D\{-\beta L\beta+3|\nabla\beta|^2\}
        \leq
        \left(\frac{\pi}{2\delta_x}\right)^2
        \left\{
        \frac{m}{2}+
       \frac{1}{2} \sqrt{
        m^2+4m\bigl(\sup_{p\in D}\ell(\rho_o(p))\bigr)\delta_x^2}
        +3
        \right\}
        +\varepsilon 
\]
for sufficiently small $\varepsilon$.
Letting \(\varepsilon\downarrow0\), we obtain the asserted bound.
Since \(D\subset B(x,\delta_x)\), we have
\[
        \rho_o(p)\leq \rho_o(x)+\delta_x,\qquad p\in D.
\]
Moreover, there exists \(c_0>0\) such that
\[
        \ell(r)\leq c_0(1+r^2),\qquad r\geq0.
\]
Consequently,
\[
        \sup_{p\in D}\ell(\rho_o(p))
        \leq
        c_0\bigl(1+(\rho_o(x)+\delta_x)^2\bigr).
\]
Substituting this bound into the preceding estimate gives
\[
        c(\beta)
        \leq
        \left(\frac{\pi}{2\delta_x}\right)^2
        \left\{
        \frac{m}{2}+
       \frac{1}{2} \sqrt{
        m^2+
        4mc_0\bigl(1+(\rho_o(x)+\delta_x)^2\bigr)\delta_x^2}
        +3
        \right\}.
\]
Changing the constant \(c_0\) if necessary yields the desired estimate.
\end{proof}

Combining Lemma \ref{lem} and Lemma \ref{lem3.3}, we have the gradient estimate for $P_{\cdot}^hf$. 

\begin{lemma}\label{s1.lem.key}
Assume Condition {\bf (C1)} and $\sup_{s\in [0,T]}\E^x |\nabla h|^2(x_s) \leq \theta_1(x)$ hold. Then for any \(T > 0\) and $f\in \mathcal{B}_b(M)$, 
\[
|\nabla P_T^h f|(x) \leq \left[ \frac{1}{\sqrt{T}} + \frac{T}{4} \sqrt{\theta_1(x)} \right] (P_T f^2)^{1/2},\quad x\in M.
\]
\end{lemma}

\begin{proof}
By the Bismut formula (Lemma \ref{lem}) with \(k(s) = \frac{T-s}{T}\), we have
\[
\langle \nabla P_T^h f, v \rangle = \frac{1}{T}\mathbb{E} \left[ \mathbb{H}_T f(x_T) \int_0^T \langle Q_s ( v), \ptr_s dB_s \rangle \right] - \frac{1}{2} \mathbb{E} \left[ \mathbb{H}_T f(x_T) \int_0^T \langle \nabla h, Q_s ( v) \rangle \frac{T-s}{T}\, ds \right].
\]
For the first term, by the Cauchy-Schwarz inequality and Itô isometry:
\[
\left|\frac{1}{T} \mathbb{E} \left[ \mathbb{H}_T f(x_T) \int_0^T \langle Q_s ( v), \ptr_s dB_s \rangle \right] \right| \leq (P_T f^2)^{1/2} \cdot \left( \int_0^T \frac{1}{T^2} ds \right)^{1/2} = \frac{1}{\sqrt{T}} (P_T f^2)^{1/2}.
\]
For the second term involving \( \nabla h \), we again apply the Cauchy–Schwarz inequality:  
\[
\begin{aligned}
&-\frac{1}{2}\mathbb{E} \left[ \mathbb{H}_T f(x_T) \int_0^T \langle \nabla h, Q_s ( v) \rangle \frac{T-s}{T}\, ds \right] \\
&\leq \frac{T}{4} (P_T f^2)^{1/2} \cdot  \sup_{s \in [0,T]} \bigl(\mathbb{E}^x |\nabla h|^2 (x_s)\bigr)^{1/2}\\
&\leq \frac{T}{4} \sqrt{\theta_1(x)} (P_T f^2)^{1/2}.
\end{aligned}
\]  
Combining both estimates yields the desired inequality.
\end{proof}

With the aid of the preceding lemmas, we are now in a position to introduce our main result in this subsection.

\begin{theorem}\label{s1.thm.2}
Assume Condition {\bf (C1)} holds for some $h\in C^1(M)$ satisfying
\begin{align*}
\sup_{s\in [0,1]}\mathbb{E}^x[ h^2(x_s)]\leq \theta_0(x),\quad  \mbox{and}\quad  \sup_{s\in [0,1]}\mathbb{E}^x[|\nabla h|^2 (x_s)] \leq \theta_1(x).
\end{align*}
Then for every fixed $T>0$, for any $f\in \mathcal{B}_b(M)$ and $x\in M$, we have
\begin{align*}
|\nabla P_Tf|(x)\leq \left[\frac{1}{\sqrt{\min\{T,1\}}}+\frac{1}{4}\sqrt{\theta_1(x)}+\sqrt{\theta_0(x)}  +\frac{1}{8}\sqrt{ \theta_0(x) \theta_1(x) } \right]\|f\|_{\infty}.
\end{align*}
\end{theorem}

\begin{proof}
Assume $T\in (0,1]$. Otherwise, we let $g=P_{T-1}f\in C_b^{\infty}(M)$ for $f\in \mathcal{B}_b(M)$ and $P_Tf=P_1g$. 
Moreover, $\|g\|_{\infty}=\|P_{T-1}f\|_{\infty}\leq \|f\|_{\infty}$. 

We begin by employing the following  variation of constants formula  to derive gradient estimates for $P_tf$:
\begin{equation}\label{Camel-id}
    P_Tf= P_T^hf+\frac{1}{2}\int_0^T P_{T-s}^h(h P_sf)\,ds, 
\end{equation}
where $h$ serves as the lower Bakry--Émery curvature bound in \eqref{c1}.  This identity is first obtained for $f\in C_N(L_{h/2})$ by Duhamel's formula for the generators $\frac12L$ and $\frac12(L-h)$, and then extended to $f\in\mathcal B_b(M)$ by bounded monotone-class approximation. Since \(h\ge0\), the Feynman--Kac weight \(\mathbb{H}_t\le1\), so no additional exponential integrability is needed for this Duhamel step.

From \eqref{Camel-id}, we obtain that for $v\in T_xM$ with $|v|=1$:
\[
\langle\nabla  P_T f, v\rangle (x)= \langle\nabla P_T^h f,  v\rangle(x) + \frac{1}{2}\int_0^T \langle\nabla P_{T-s}^h (h P_{s} f), v\rangle(x) \, ds.
\]
By Lemma \ref{s1.lem.key}, we see that
\begin{align*}
\langle\nabla  P_T f, v\rangle (x)&= \langle\nabla P_T^h f,  v\rangle(x) + \frac{1}{2}\int_0^T \langle\nabla P_{T-s}^h (h P_{s} f), v\rangle(x) \, ds\\
&\leq  \lf( \frac{1}{\sqrt{T}}+\frac{1}{4}\sqrt{\theta_1(x)}T\r)(P_Tf^2)^{1/2}\\
&\qquad +\frac{1}{2} \int_0^T\lf( \frac{1}{\sqrt{T-s}}+\frac{1}{2} \sqrt{\theta_1(x)}(T-s)\r)(P_{T-s} (hP_{s}f)^2)^{1/2}\, ds\\
&\leq  \lf( \frac{1}{\sqrt{T}}+\frac{1}{4}\sqrt{\theta_1(x)}T\r)\|f\|_{\infty}+\frac{\sqrt{\theta_0(x)}}{2} \int_0^T\lf( \frac{1}{\sqrt{T-s}}+\frac{1}{2}\sqrt{\theta_1(x)}(T-s)\r) ds\,  \|f\|_{\infty}\\
&= \lf[ \frac{1}{\sqrt{T}}+\frac{1}{4}\sqrt{\theta_1(x)}T+\sqrt{\theta_0(x)T}+\frac{T^2}{8}\sqrt{\theta_1(x)\theta_0(x)}\r]\|f\|_{\infty}.
\end{align*}

\end{proof}
\begin{remark}
In particular, when the lower bound $h$ is a constant, the main result reduces to the following statement.
Assume Condition \textbf{(C1)} holds with a nonnegative constant $h=K$.
For a fixed $T>0$ and $x\in M$, we have
\begin{align}\label{ineq-constant-Ricci}
|\nabla P_Tf|(x)\leq \Bigl(K +\frac{1}{\sqrt{\min\{T,1\}}} \Bigr)\,\|f\|_{\infty}.
\end{align}
Under the same hypothesis, the classical Bismut formula for $\nabla P_tf$, together with $|Q_s|\le \e^{Ks/2}$ when $\Ric_Z\ge -K$, yields
\begin{align*}
|\nabla P_Tf|(x)\leq \frac{\|f\|_{\infty}}{\sqrt{\int_0^T \e^{-Ks}\, ds}},
\end{align*}
after the usual energy-minimizing choice of $k$. 
Comparing this with inequality \eqref{ineq-constant-Ricci}, we observe that the bound derived from the Bismut formula is sharper.
This indicates that although the approach adopted in this paper can readily handle pointwise lower Ricci bounds, it entails a slight loss of sharpness in the estimate.
\end{remark}

\subsection{Explicit Examples under Radial Curvature Growth}

Let \(g\in C^1([1,\infty))\) be a positive non-decreasing function. Assume
that
\begin{equation}\label{condition-g}
        \sup_{r\geq0}
        \left\{
        \frac{1+\sqrt{\ell(r)}\,r}{g(r^2+1)}
        +
        \frac{r^2}{g(r^2+1)^2}
        \right\}
        <\infty .
\end{equation}
Let
\begin{equation}\label{form-f}
        f(r)
        =
        M_0\exp\left(a_0\int_1^r \frac{ds}{g(s)}\right),
        \qquad
        M_0>0,\quad a_0>0 .
\end{equation}
We assume that the function \(h\) appearing in Condition {\bf(C1)} is given by
\[
        h(x)=\sqrt{f(\rho_o(x)^2+1)}.
\]

For \(r>1\), we have
\[
        f'(r)=\frac{a_0 f(r)}{g(r)}
\]
and
\[
        f''(r)
        =
        \frac{a_0^2 f(r)}{g(r)^2}
        -
        \frac{a_0 f(r)g'(r)}{g(r)^2}
        \leq
        \frac{a_0^2 f(r)}{g(r)^2},
\]
where we used \(g'(r)\geq0\).

For brevity, we write \(\rho_o(\cdot)=\rho(o,\cdot)\). By Itô's formula for
the radial process and the Laplacian comparison estimate
\[
        L\rho_o^2
        \leq
        m+\sqrt{m^2+4m\ell(\rho_o)\rho_o^2},
\]
we obtain, outside the cut-locus and hence everywhere by the usual Calabi
argument,
\[
\begin{aligned}
d f(\rho_o(x_t)^2+1)
&\overset{\text{m}}{\leq}
\frac12
\left[
        f'(\rho_o(x_t)^2+1)L\rho_o(x_t)^2
        +
        f''(\rho_o(x_t)^2+1)|\nabla\rho_o(x_t)^2|^2
\right]dt
\\
&\leq
\frac12 f(\rho_o(x_t)^2+1)
\Biggl[
        \frac{a_0\left(m+\sqrt{m^2+4m\ell(\rho_o(x_t))\rho_o(x_t)^2}\right)}
             {g(\rho_o(x_t)^2+1)}
        +
        \frac{4a_0^2\rho_o(x_t)^2}
             {g(\rho_o(x_t)^2+1)^2}
\Biggr]dt .
\end{aligned}
\]
By \eqref{condition-g}, there exists a constant \(C_1>0\), depending only on
\(a_0,m,\ell\) and \(g\), such that
\[
        d f(\rho_o(x_t)^2+1)
        \overset{\text{m}}{\leq}
        C_1 f(\rho_o(x_t)^2+1)\,dt .
\]

Since the preceding differential inequality is understood modulo the
differential of a local martingale, we justify the passage to expectations by
localization. Let
\[
        \tau_n:=\inf\{s\geq0:\rho_o(x_s)\geq n\}.
\]
After possibly replacing \(\tau_n\) by a smaller localizing sequence, the
local martingale part stopped at \(\tau_n\) is a true martingale. Applying the
preceding inequality to \(t\wedge\tau_n\), we get
\[
\mathbb E^x f(\rho_o(x_{t\wedge\tau_n})^2+1)
\leq
f(\rho_o(x)^2+1)
+
C_1\int_0^t
\mathbb E^x f(\rho_o(x_{s\wedge\tau_n})^2+1)\,ds .
\]
By Gronwall's inequality,
\[
\mathbb E^x f(\rho_o(x_{t\wedge\tau_n})^2+1)
\leq
\e^{C_1t}f(\rho_o(x)^2+1).
\]
Letting \(n\to\infty\) and using Fatou's lemma yields
\begin{equation}\label{esti-eq-1}
        \mathbb E^x f(\rho_o(x_t)^2+1)
        \leq
        \e^{C_1t}f(\rho_o(x)^2+1).
\end{equation}

Moreover,
\[
\begin{aligned}
\left|\nabla\sqrt{f(\rho_o(\cdot)^2+1)}\right|^2(x)
&=
\frac{\left|f'(\rho_o(x)^2+1)\nabla\rho_o(\cdot)^2\right|^2(x)}
     {4f(\rho_o(x)^2+1)}
\\
&=
a_0^2\rho_o(x)^2
\frac{f(\rho_o(x)^2+1)}
     {g(\rho_o(x)^2+1)^2}
\\
&\leq
C_2 f(\rho_o(x)^2+1)
\end{aligned}
\]
for some constant \(C_2>0\). Consequently, by \eqref{esti-eq-1},
\begin{equation}\label{esti-eq-2}
        \mathbb E^x
       \l[ \left|\nabla\sqrt{f(\rho_o(\cdot)^2+1)}\right|^2(x_t)\r]
        \leq
        C_2\e^{C_1t}f(\rho_o(x)^2+1).
\end{equation}

Combining \eqref{esti-eq-1} and \eqref{esti-eq-2} with
Theorem~\ref{s1.thm.2}, we obtain the following estimate.

\begin{theorem}\label{h-theorem}
Assume Condition {\bf(C1)} holds with
\[
        h(x)=\sqrt{f(\rho_o(x)^2+1)},
\]
where \(f\) is defined by \eqref{form-f}. Then, for every \(T>0\),
\(x\in M\), and \(u\in\mathcal B_b(M)\),
\[
\begin{aligned}
|\nabla P_Tu|(x)
\leq
\Bigg[
&
\frac1{\sqrt{\min\{T,1\}}}
+
\left(\frac{\sqrt{C_2}}4+1\right)
\e^{\frac{1}{2}C_1} h(x)
+
\frac{\sqrt{C_2}}8
\e^{C_1}h(x)^2
\Bigg]\|u\|_\infty .
\end{aligned}
\]
\end{theorem}

In particular, suppose that the curvature lower bound is of the form
\[
        \ell(r)=K_0(1+r^2)^\gamma,
        \qquad K_0>0.
\]
If \(\gamma\in[0,1)\), we may take
\[
        g(r)=r^{\frac{\gamma+1}{2}}.
\]
Then
\[
        f(r)
        =
        M_0
        \exp\left(
        \frac{2a_0}{1-\gamma}
        \left(r^{\frac{1-\gamma}{2}}-1\right)
        \right).
\]
If \(\gamma=1\), we may take \(g(r)=r\), and then
\[
        f(r)=M_0r^{a_0}.
\]

Thus we obtain the following corollary.

\begin{corollary}\label{examples}
Assume Condition {\bf(C1)} holds with
\(
        \ell(\rho_o)=K_0(1+\rho_o^2)^\gamma,    \ 
        \gamma\in[0,1).
\)
Suppose
\[
        h(x)
        =
        C\exp\left(
        a(1+\rho_o(x)^2)^{\frac{1-\gamma}{2}}
        \right)
\]
for some constants \(K_0,\,C,\,a>0\). Then there exist positive constants
\(C_1=C_1(a,K_0,m,\gamma)\) and
\(C_2=C_2(a,K_0,\gamma)\) such that for every \(T>0\),
\(x\in M\), and \(u\in\mathcal B_b(M)\),
\[
\begin{aligned}
|\nabla P_Tu|(x)
\leq
\Biggl[
&
\frac1{\sqrt{\min\{T,1\}}}
+
C\left(\frac{\sqrt{C_2}}4+1\right)
\,
\exp\left(
        \frac{1}{2}C_1+a(1+\rho_o(x)^2)^{\frac{1-\gamma}{2}}
        \right)
\\
&\quad
+
\frac{C^2\sqrt{C_2}}8
\exp\left(
       C_1+ 2a(1+\rho_o(x)^2)^{\frac{1-\gamma}{2}}
        \right)
\Biggr]\, \|u\|_\infty .
\end{aligned}
\]

If Condition {\bf(C1)} holds with
\(
        \ell(\rho_o)=K_0(1+\rho_o^2)
\)
and
\[
        h(x)=C(1+\rho_o(x)^2)^{p/2}
\]
for some constants \(K_0,\, C,\, p>0\), then there exist positive constants
\(C_1=C_1(K_0,m,p)\) and \(C_2=C_2(K_0,p)\) such that  for every \(T>0\),
\(x\in M\), and \(u\in\mathcal B_b(M)\),
\[
\begin{aligned}
|\nabla P_Tu|(x)
\leq
\Biggl[
&
\frac1{\sqrt{\min\{T,1\}}}
+
\left(\frac{\sqrt{C_2}}4+1\right)
\e^{\frac{1}{2}C_1}\,
C(1+\rho_o(x)^2)^{p/2}
+
\frac{\sqrt{C_2}}8
\e^{C_1}
C^2(1+\rho_o(x)^2)^p
\Biggr]\, \|u\|_\infty .
\end{aligned}
\]
\end{corollary}

\section{Gradient Estimates for Neumann Semigroups with Non-Convex Boundary}

\subsection{A Boundary Local-Time Estimate}

Before proving the main gradient estimate when the boundary is non-convex, we establish an auxiliary estimate involving the local time of the reflecting diffusion.  To this end, we introduce the following set,

\begin{definition}\label{def-DM}
We denote by \(\mathcal D(M)\) the class of all positive functions
\(\phi\in C^2(M^\circ)\cap C^1(M)\) such that
\[
        1\leq \inf_M\phi\leq \sup_M\phi<\infty,
\]
and such that \(L\log\phi\) is well defined in \(M^\circ\) and
\(N\log\phi\) is well defined on \(\partial M\). Moreover, we assume that
\[
        (L\log\varphi)^-, \, |\nabla\log\varphi|^2\in\mathcal K(M),\ \ \mbox{and}
\ \displaystyle
\varlimsup_{\rho_o(x)\to\infty}
\frac{(L\log\phi)^-+|\nabla\log\phi|^2}{1+\rho_o(x)^2}<\infty.
\]
\end{definition}
  This condition is automatically satisfied, for instance, if
\((L\log\phi)^-+|\nabla\log\phi|^2\) is bounded, or if it depends only on
\(\rho_{\partial}\) and has at most quadratic growth in
\(\rho_{\partial}\), where $\rho_{\partial}$ is the distance to the boundary $\partial M$. 

\begin{lemma}\label{lem-local-time}
Let \(\sigma\in C(\partial M)\) be non-negative. Suppose that there exists
\(\phi\in\mathcal D(M)\) such that
\[
        N\log\phi\geq \sigma
        \quad \hbox{on } \partial M .
\]
Let \(\psi\) be a non-negative measurable function satisfying
\[
        -L\log\phi+2|\nabla\log\phi|^2\leq \psi
        \quad \hbox{on } M^\circ .
\]
Then, for every \(t\geq0\) and \(x\in M\),
\[
\mathbb{E}^x \left[
\phi^{-2}(x_t)
\exp \left( \int_0^t \sigma(x_s)\,dl_s \right)
\exp \left( - \int_0^t \psi(x_s)\,ds \right)
\right]
\leq
\phi^{-2}(x).
\]
\end{lemma}

\begin{proof}
Set \(g=\log\phi\) and define
\[
Y_t
=
\phi^{-2}(x_t)
\exp \left( \int_0^t \sigma(x_s)\,dl_s \right)
\exp \left( - \int_0^t \psi(x_s)\,ds \right).
\]
Equivalently,
\[
Y_t
=
\e^{-2g(x_t)}
\exp \left(
        \int_0^t \sigma(x_s)\,dl_s
        -
        \int_0^t \psi(x_s)\,ds
\right).
\]
By Itô's formula, modulo the differential of a local martingale,
\[
\begin{aligned}
dY_t
\overset{m}{=}
&\;
\e^{\int_0^t \sigma(x_s)\,dl_s-\int_0^t\psi(x_s)\,ds}
\left[
        \frac12 L\phi^{-2}-\psi\phi^{-2}
\right](x_t)\,dt
\\
&\quad
+
\e^{\int_0^t \sigma(x_s)\,dl_s-\int_0^t\psi(x_s)\,ds}
\left[
        \frac12 N\phi^{-2}+\sigma\phi^{-2}
\right](x_t)\,dl_t .
\end{aligned}
\]
Since \(\phi^{-2}=\e^{-2g}\), we have
\[
        \frac12 L\phi^{-2}
        =
        \phi^{-2}\left(-Lg+2|\nabla g|^2\right).
\]
Hence
\[
        \frac12 L\phi^{-2}-\psi\phi^{-2}
        =
        \phi^{-2}
        \left(
        -L\log\phi+2|\nabla\log\phi|^2-\psi
        \right)
        \leq0.
\]
Moreover,
\[
        \frac12 N\phi^{-2}
        =
        -\phi^{-2}N\log\phi.
\]
Therefore
\[
        \frac12 N\phi^{-2}+\sigma\phi^{-2}
        =
        \phi^{-2}\left(\sigma-N\log\phi\right)
        \leq0
        \quad \hbox{on } \partial M .
\]
Consequently \(Y_t\) is a non-negative local supermartingale.

Let \(\tau_n\) be a localizing sequence. Then
\[
        \mathbb E^x Y_{t\wedge\tau_n}
        \leq
        Y_0
        =
        \phi^{-2}(x).
\]
Letting \(n\to\infty\) and using Fatou's lemma, we obtain
\[
        \mathbb E^xY_t
        \leq
        \liminf_{n\to\infty}
        \mathbb E^xY_{t\wedge\tau_n}
        \leq
        \phi^{-2}(x).
\]
This proves the desired estimate.
\end{proof}

\subsection{Conformal Reduction to the Convex Boundary Case}
We begin by introducing a curvature-dimension condition which allows the lower bound to decay in a controlled manner with respect to distance.

\begin{mdframed}  
\textbf{Condition (C2)}:  
There exists a non‑negative function $\ell \in C^1(\R^+)$ such that $\ell(r)\leq c_0(1+r^2)$ for some nonnegative constant $c_0$, and  
\begin{align}
	\Ric_Z^{m,d}(p):=\Ric(p)-\nabla Z^\flat(p)-\frac{Z\otimes Z}{m-d}(p)\geq -\ell\bigl(\rho_o( p)\bigr). 
\end{align}
Moreover, there exist non‑negative functions $h\in C^1(M)$ and $\sigma\in C^1(\partial M)$ such that outside the cut locus  
 \begin{align*}
&\Ric_Z\geq -h, \qquad \II\geq -\sigma . 
\end{align*}
\end{mdframed}

The next result shows that Condition {\bf(C2)}, together with a suitable
boundary condition on the conformal factor \(\phi\), allows us to reduce the
non-convex boundary case to the convex one. We use the conformally changed
metric
\[
        g'=\phi^{-2}g
\]
and denote by \(\rho'\) the corresponding geodesic distance. The conformal
transformation formulas used below are standard; see, for example,
\cite[Section~3.2]{Wbook14}.

\begin{lemma}\label{lem3-1}
Assume Condition {\bf(C2)}. Suppose that there exists
\(\phi\in\mathcal D(M)\) such that
\[
        N\log\phi\geq \sigma
        \quad \hbox{on }\ \partial M .
\]
Let
\[
        \tau_n:=\inf\{s\geq0:\rho'(x,x_s)\geq n\}.
\]
Then there exists a sequence of absolutely continuous adapted processes
\(\{k_n\}_{n\geq1}\) such that
\[
        k_n(0)=1,
        \qquad
        k_n(s)=0 \quad \hbox{for } s\geq t\wedge\tau_n,
\]
and such that, for some constant \(c>0\),
\[
        \sup_{n\geq1}
        \mathbb E^x\left[\int_0^t |\dot k_n(s)|^2\,ds\right]
        \leq
        \frac1t+c.
\]
\end{lemma}

\begin{proof}
We first show that the conformal change removes the boundary
non-convexity. Under the conformal metric \(g'=\phi^{-2}g\), the second
fundamental form satisfies
\[
        \II'
        =
        \phi^{-1}\bigl(\II+(N\log\phi)g_{\partial M}\bigr).
\]
Since \(\II\geq-\sigma\) and \(N\log\phi\geq\sigma\), it follows that
\[
        \II'\geq0.
\]
Thus the boundary is convex with respect to \(g'\).

Next we express the generator in the conformally changed metric. Since
\[
        \Delta'
        =
        \phi^2\bigl(\Delta-(d-2)\langle\nabla\log\phi,\nabla\cdot\rangle\bigr),
\]
we have
\[
        L=\Delta+Z
        =
        \phi^{-2}(\Delta'+Z'),
\]
where
\[
        Z'=\phi^2\bigl(Z+(d-2)\nabla\log\phi\bigr).
\]
Equivalently,
\[
        L=\phi^{-2}L',
        \qquad
        L':=\Delta'+Z'.
\]

We now check the curvature condition for \(L'\). Let \(\dot\gamma\) be a
unit vector with respect to \(g'\). Then \(|\dot\gamma|=\phi\) with respect
to \(g\). By the standard conformal transformation formula for the
Bakry--Emery curvature, there exists a constant \(C=C(m,d)>0\) such that
\[
\begin{aligned}
&(\Ric^{Z'})'(\dot\gamma,\dot\gamma)
-
\frac{\langle Z',\dot\gamma\rangle'^2}{2(m-d)}
\geq
-\phi^2\ell(\rho_o)
+
\frac12 L\phi^2
-
C|\nabla\phi|^2 .
\end{aligned}
\]
Using
\[
        \frac12L\phi^2
        =
        \phi^2\left(L\log\phi+2|\nabla\log\phi|^2\right),
        \qquad
        |\nabla\phi|^2=\phi^2|\nabla\log\phi|^2,
\]
we obtain
\[
\begin{aligned}
&(\Ric^{Z'})'(\dot\gamma,\dot\gamma)
-
\frac{\langle Z',\dot\gamma\rangle'^2}{2(m-d)}
\\
&\qquad\geq
-\phi^2
\left[
        \ell(\rho_o)
        +(L\log\phi)^-
        +C|\nabla\log\phi|^2
\right].
\end{aligned}
\]
Since \(1\leq \inf_M\phi\leq\sup_M\phi<\infty\), the distances \(\rho_o\) and
\(\rho'_o\) are equivalent:
\[
        \frac{1}{\|\phi\|_\infty}\rho_o
        \leq
        \rho'_o
        \leq
        \rho_o .
\]
By the growth assumption in Condition {\bf(C2)} and the defining properties
of \(\mathcal D(M)\), the lower bound above has the same quadratic growth
control, in terms of \(\rho'_o\), as required in Condition {\bf(C1)}. Hence
the conformal operator \(L'\) satisfies the convex-boundary curvature
condition with dimension \(2m-d\).

Therefore Lemma~\ref{lem3.3} can be applied to the \(L'\)-diffusion stopped
upon leaving the \(g'\)-ball. Since \(L=\phi^{-2}L'\) and \(\phi\) is bounded
above and below by positive constants, the corresponding \(L\)-diffusion is a
bounded time change of the \(L'\)-diffusion. Consequently the stopping times
defined by \(g'\)-balls and the energy estimates for the cut-off processes
are preserved up to a multiplicative constant. Thus there exists a sequence
\(\{k_n\}_{n\geq1}\) satisfying
\[
        k_n(0)=1,
        \qquad
        k_n(s)=0 \quad \hbox{for } s\geq t\wedge\tau_n,
\]
and
\[
        \sup_{n\geq1}
        \mathbb E^x\left[\int_0^t |\dot k_n(s)|^2\,ds\right]
        \leq
        \frac1t+c
\]
for some constant \(c>0\). The proof is complete.
\end{proof}

\subsection{A Bismut Formula for Feynman--Kac Semigroups after Conformal Reduction}
Assume that there exists a non-negative function $\psi\in C^1(M)$ such that 
$- L \log \phi + 2|\nabla \log \phi|^2\leq \psi$.  Let 
$$V = \frac12 (h +\psi) $$
in what follows. 
We now derive a pointwise gradient estimate for \(P_t^V f\) under the integrability assumptions on \(V\) and its gradient.

\begin{lemma}\label{lem0-0}
Assume Condition {\bf (C2)} holds for some non-negative functions \( h \in C^1(M) \) and \( \phi \in \mathcal{D}(M) \)  such that  \( N \log \phi \geq \sigma \) on the boundary. If there exists a non-negative function $\psi\in C^1(M)$ and $\alpha_1\in C(M)$ such that 
$- L \log \phi + 2|\nabla \log \phi|^2\leq \psi$ and
\[
\sup_{s \in [0,T]} \mathbb{E}^x \left[ \left| \nabla (h+\psi)  \right|^2 (x_s) \right] \leq \alpha_1(x),
\]  
then  there exists \( Q_t : T_xM \to T_{x_t^x}M \) such that
\begin{align*}
\langle Q_t(v), N \rangle |_{\partial M} =0, \ \  v\in T_x M\ \ \mbox{and}\quad |Q_t|\leq \e^{\frac{1}{2}\int_0^t h(x_s)\,ds +\frac{1}{2}\int_0^t \sigma(x_s)\,dl_s },
\end{align*} 
and for $k\in C^1([0,T])$ such that $k(0)=1$ and $k(T)=0$, 
\[
\langle \nabla P_T^V f, v \rangle = 
-\mathbb{E} \l[ \mathbb{V}_T f(x_T) \int_0^T \langle Q_s (\dot{k}(s) v), \ptr_s dB_s \rangle \r] 
- \mathbb{E} \l[ \mathbb{V}_T f(x_T) \int_0^T \langle \nabla V, Q_s (k(s) v) \rangle \, ds \r].
\]  
\end{lemma}

\begin{proof}
We begin with the discussion as in the proof of Lemma \ref{lem}. For $v\in T_xM$ and 
$k_n\in L^{1,2}([0,T]; \mathbb{P})$ a sequence of non-negative processes as in Lemma  \ref{lem3-1} such that $k_n(0)=1$ and
$k_n(s)=0$ for $s\geq T\wedge \tau_{B'(x,n)}$,
\begin{align}\label{Bismut-F}
    \ed P_T^V f(v)&= -\E\l[ \mathbb{V}_{T} f(x_{T})\int_0^{T} \langle Q_t(\dot{k}_n(t)v), /\!/ _t \,dB_t\rangle\r]
    -\E\l[\mathbb{V}_T f(x_T)\int_0^{T} k_n(t) \, \ed V (Q_t(v))\,dt \r] \notag\\
&\leq  \|f\|_{\infty} \,
\mathbb{E} \Biggl[ \e^{-\int_0^T V(x_s) ds} \e^{\frac{1}{2}\int_0^T h(x_s) ds + \frac{1}{2}\int_0^T \sigma(x_s) d l_s} \notag \\
&\qquad\qquad \qquad \times \Bigl| \int_0^T 
\langle \e^{-\frac12\int_0^s h(x_s) ds - \frac12\int_0^s \sigma(x_s) d l_s} Q_s (\dot{k}(s) v), \ptr_s dB_s \rangle \Bigr| \Biggr] \notag \\
& \qquad + \|f\|_{\infty} \,
\mathbb{E} \Biggl[ \e^{-\int_0^T V(x_s) ds} \e^{\frac{1}{2}\int_0^T h(x_s) ds + \frac{1}{2}\int_0^T \sigma(x_s) d l_s} \notag \\
&\qquad\qquad \qquad \times  \int_0^T 
\e^{-\frac12\int_0^s h(x_s) ds - \frac12\int_0^s \sigma(x_s) d l_s} Q_s (\dot{k}(s) v) |\nabla V|(x_s)\, ds  \Biggr]. 
\end{align} 
We first see that  
\begin{align*}
\e^{-\int_0^T V(x_s) ds} \e^{\frac{1}{2}\int_0^T h(x_s) ds + \frac{1}{2}\int_0^T \sigma(x_s) d l_s} 
&\leq \e^{\frac{1}{2}\int_0^{T} \sigma (x_s) d l_s} \e^{-\frac{1}{2}\int_0^T\psi(x_s)\,d s},
\end{align*}
and by Lemma \ref{lem3-1}, we see that
\begin{align*}
\E[|\mathbb{V}_T Q_{t\wedge \tau_{B'(x,n)}}|^2]^{1/2}\leq \frac{\|\phi\|_{\infty}}{\phi(x)}. 
\end{align*}
Then, we have
\begin{align*}
    |\ed P_T^V f|&\leq \l|\E\l[ \mathbb{V}_T f(x_T)\int_0^{T} \langle Q_t(\dot{k}_n(t)v), /\!/ _t \,dB_t\rangle\r]\r|+
 \l|\E\l[\mathbb{V}_T f(x_T)\int_0^{T} k_n(t) \, \ed V (Q_t(v))\, dt \r]\r|\\
   &\leq \frac{\|\phi\|_{\infty}}{\phi(x)}  \E\l[\int_0^{T}|\dot{k}_n(t)|^2 \,d t\r]^{1/2}\|f\|_{\infty}+\frac{\|\phi\|_{\infty}}{\phi(x)}T \sup_{s\in [0,T]}\E\Big[ |\nabla V|^2(x_s)\Big] \, \|f\|_{\infty}.
\end{align*} 
By Lemma \ref{lem3-1}, we know that 
\begin{align*}
    |\ed P_T^V f|&\leq  \frac{\|\phi\|_{\infty}}{\phi(x)} \|f\|_{\infty} \l[\Bigl(\frac{1}{T}+c\Bigr)^{1/2}+ T \sup_{s\in [0,T]}\E |\nabla V|^2(x_s)\r],
    \end{align*}
    which implies that
    $|\nabla P_{\cdot}^Vf|$ is bounded in $[\epsilon, T]\times M$ for small $\epsilon>0$.
    This estimate justifies the
passage to the limit and the removal of the localization with the same reason as in the proof of Lemma \ref{lem}.
     
\end{proof}

From the Bismut formula above, we obtain the following gradient estimate for \(P^V_t f\).

\begin{lemma}\label{lem0}
Assume Condition {\bf (C2)} holds for some non-negative functions \( h \in C^1(M) \) and \( \phi \in \mathcal{D}(M) \)  such that  \( N \log \phi \geq \sigma \) on the boundary. If  there exists non-negative function $\psi\in C^1(M)$ such that   $- L \log \phi + 2|\nabla \log \phi|^2\leq \psi$ and
\[
\sup_{s \in [0,T]} \mathbb{E}^x \left[ \left| \nabla (h+\psi)  \right|^2 (x_s) \right] \leq \alpha_1(x),
\]  
then for any \(T>0\), 
\[
|\nabla P_T^V f|(x) \leq \left(\frac{1}{\sqrt{T}}  + \frac{T}{4} \sqrt{\alpha_1(x)} \right) 
\frac{\|\phi\|_{\infty}}{\phi(x)} \,  \bigl(P_T f^2\bigr)^{1/2}(x).
\]
\end{lemma}

\begin{proof}
We start from the Bismut formula for $\nabla P_T^V f$ (see Lemma \ref{lem0-0})
with  \(k(s)=\frac{T-s}{T},\, s\in [0,T]\) and estimate the two terms separately.

For the stochastic integral term, Cauchy–Schwarz and Itô isometry give  
\[
\begin{aligned}
&-\mathbb{E} \Bigl[ \mathbb{V}_T f(x_T) \int_0^T \langle Q_s (\dot{k}(s) v), \ptr_s dB_s \rangle \Bigr] \\
&\leq   \bigl(P_T f^2\bigr)^{1/2}(x)\,
\mathbb{E} \Biggl[ \e^{-2\int_0^T V(x_s) \,ds} \Bigl(\int_0^T 
\langle  Q_s (\dot{k}(s) v), \ptr_s dB_s \rangle \Bigr)^2 \Biggr] ^{1/2}\\
&\leq  \bigl(P_T f^2\bigr)^{1/2}(x)\,
\mathbb{E} \Bigl[ \int_0^T \e^{ -\int_0^t \psi(x_s) \, ds} 
\e^{\int_0^t \sigma(x_s) \,d l_s}\, dt \Bigr]^{1/2} \frac{1}{T} .
\end{aligned}
\]  
Using  Lemma \ref{lem-local-time}, the expectation above is bounded by 
\[
-\mathbb{E} \Bigl[ \mathbb{V}_T f(x_T) \int_0^T \langle Q_s (\dot{k}(s) v), \ptr_s dB_s \rangle \Bigr]
\leq \frac{1}{\sqrt{T}} \cdot \frac{\|\phi\|_{\infty}}{\phi(x)} \cdot \bigl(P_T f^2\bigr)^{1/2}(x).
\]

For the drift term involving \(\nabla V\), Cauchy–Schwarz yields  

\[
\begin{aligned}
&-\mathbb{E} \Bigl[ \mathbb{V}_T f(x_T) \int_0^T \langle \nabla V, Q_s (k(s) v) \rangle ds \Bigr] \\
&\leq \frac{T}{2}\cdot 
\frac{\|\phi\|_{\infty}}{\phi(x)} \,
\sup_{s\in[0,T]} \bigl(\mathbb{E} |\nabla V|^2 (x_s)\bigr)^{1/2} \, \bigl(P_T f^2\bigr)^{1/2}(x).
\end{aligned}
\]  
By the assumed bound on \(\nabla V\), this becomes  
\[
-\mathbb{E} \Bigl[ \mathbb{V}_T f(x_T) \int_0^T \langle \nabla V, Q_s (k(s) v) \rangle ds \Bigr]  
\leq \frac{T}{4} \cdot \frac{\|\phi\|_{\infty}}{\phi(x)} \sqrt{ \alpha_1(x)} \cdot  \,\bigl(P_T f^2\bigr)^{1/2}(x).
\]
Adding the two estimates gives  
\[
|\nabla P_T^V f|(x) \leq 
\left( \frac{1}{\sqrt{T}} + \frac{T}{4}\sqrt{ \alpha_1(x)} \right) 
\frac{\|\phi\|_{\infty}}{\phi(x)} \bigl(P_T f^2\bigr)^{1/2}(x),
\]
which is the desired inequality.
\end{proof}

\subsection{Gradient estimate for the Neumann semigroup}

Finally, we combine the preceding lemmas with  the Duhamel formula, to obtain a pointwise gradient bound for the original Neumann semigroup \(P_t\).

\begin{theorem} \label{s1.thm.3}
Assume Condition {\bf (C2)} holds for some non-negative functions \( h \in C^1(M)  \) and \( \phi \in \mathcal{D}(M) \)  such that  \( N \log \phi \geq \sigma \) on the boundary.  If there exists some non-negative function  $\psi\in C^1(M)$ such that  $- L \log \phi + 2|\nabla \log \phi|^2\leq \psi$,  and there exist functions $\alpha_0, \alpha_1\in \mathcal{B}(M)$ such that
\[
\begin{aligned}
&\sup_{s\in [0,\, 1]} \mathbb{E}^x\Bigl[| h +\psi |^2(x_s) \Bigr] \leq \alpha_0(x), \quad \sup_{s\in [0,\, 1]} \mathbb{E}^x\Bigl[ \bigl| \nabla (h +\psi)\bigr|^2(x_s) \Bigr] \leq \alpha_1(x),
\end{aligned}
\]
then for any \( T > 0 \) and \( x \in M \),
\[
|\nabla P_T f|(x) \leq 
\frac{\|\phi\|_{\infty}}{\phi(x)} 
\l[\frac{1}{\sqrt{\min\{T,1\}}} + \frac14\sqrt{\alpha_1(x)} + \sqrt{\alpha_0(x)} 
+  \frac{1}{8}\sqrt{ \alpha_0(x) \alpha_1(x) } \r] \|f\|_{\infty}.
\]
\end{theorem}

\begin{proof} 
As explained at the beginning of proof of Theorem \ref{s1.thm.2}, we assume $0<T\leq 1$.
We start from the Duhamel formula,
\begin{equation*}
P_T f = P_T^V f + \int_0^T P_{T-s}^V (V P_s f) \, ds,
\end{equation*}
where \( V = \frac12 \l(h +\psi\r) \).

Differentiating along a unit vector \( v \in T_x M \) gives  
\[
\langle \nabla P_T f, v \rangle (x) = 
\langle \nabla P_T^V f, v \rangle (x) 
+ \int_0^T \langle \nabla P_{T-s}^V (V P_s f), v \rangle (x) \, ds.
\]
Applying Lemma \ref{lem0} to the first term and to the second term with time parameter $r=T-s$, then changing variables $r\mapsto s$, yields
\begin{align}\label{Bis-ineq10}
\langle \nabla P_T f, v \rangle (x) 
&\leq \frac{\|\phi\|_{\infty}}{\phi(x)}
\Bigl( \frac{1}{\sqrt{T}} + \frac14\sqrt{\alpha_1(x)}\, T \Bigr) \bigl(P_T f^2\bigr)^{1/2}  \notag \\
&\quad + \frac{\|\phi\|_{\infty}}{\phi(x)} 
\int_0^T \Bigl( \frac{1}{\sqrt{s}} + \frac14 \sqrt{\alpha_1(x)}\, s \Bigr) 
\bigl(P_s (V P_{T-s} f)^2\bigr)^{1/2} \, ds.
\end{align}
Since the semigroup is contractive on \(L^\infty\),
\[
\bigl(P_T f^2\bigr)^{1/2} \leq \|f\|_\infty , \qquad 
\bigl(P_s (V P_{T-s} f)^2\bigr)^{1/2} \leq \bigl(P_s V^2\bigr)^{1/2} \|f\|_\infty .
\]
By the definition  \(V=\frac{1}{2}(h+\psi)\) and the hypothesis on \(h+\psi\),
\[
\sup_{s\in [0,T]}\bigl(P_s V^2\bigr)^{1/2}(x) \leq \tfrac12\sqrt{\alpha_0(x)} .
\]
Substituting these bounds into \eqref{Bis-ineq10} we obtain  
\[
\begin{aligned}
\langle \nabla P_T f, v \rangle (x) 
&\leq \frac{\|\phi\|_{\infty}}{\phi(x)} 
\Bigl( \frac{1}{\sqrt{T}} + \frac14\sqrt{\alpha_1(x)}\, T \Bigr) \|f\|_\infty \\
&\quad + \frac{\|\phi\|_{\infty}}{2\phi(x)} \sqrt{\alpha_0(x)} 
\l[\int_0^T \Bigl( \frac{1}{\sqrt{s}} + \frac14 \sqrt{\alpha_1(x)}\, s \Bigr) \, ds\r] \|f\|_\infty.
\end{aligned}
\]
Computing the integral,
\[
\int_0^T \Bigl( \frac{1}{\sqrt{s}} + \frac14 \sqrt{\alpha_1(x)}\, s \Bigr) \, ds 
= 2\sqrt{T} + \frac18 \sqrt{\alpha_1(x)}\, T^2 .
\]
Therefore,
\[
\begin{aligned}
\langle \nabla P_T f, v \rangle (x) 
&\leq \frac{\|\phi\|_{\infty}}{\phi(x)} 
\l[ \frac{1}{\sqrt{T}} + \frac14\sqrt{\alpha_1(x)}\, T + \sqrt{\alpha_0(x)T} 
+ \frac{T^2}{8} \sqrt{\alpha_0(x)\alpha_1(x)} \r] \|f\|_\infty.
\end{aligned}
\]
Taking the supremum over unit vectors \(v\) gives the stated estimate.
\end{proof}

\subsection{Explicit Estimates via Boundary Distance}

In this subsection, we impose an additional geometric condition near the
boundary, which allows us to construct the conformal factor \(\phi\) in terms
of the boundary distance function \(\rho_{\partial}\).

\begin{mdframed}
\textbf{Condition (C3)}: There exist non-negative constants \(\theta\) and
\(\sigma\) such that
\[
        -\sigma\leq \II\leq \theta
        \quad \hbox{on } \partial M .
\]
Moreover, there exists a positive constant \(r_0\) such that, on
\[
        \partial_{r_0}M:=\{x\in M:\rho_{\partial}(x)<r_0\},
\]
the boundary distance function \(\rho_{\partial}\) is smooth, \(|Z|\) is
bounded, and
\[
        \operatorname{Sect}\leq k
\]
for some positive constant \(k\).
\end{mdframed}

Let
\[
        J(t):=
        \cos(\sqrt{k}\,t)
        -
        \frac{\theta}{\sqrt{k}}\sin(\sqrt{k}\,t),
        \qquad t\geq0.
\]
Let \(J^{-1}(0)\) denote the first zero of \(J\), with the convention
\(J^{-1}(0)=\infty\) if \(J(t)>0\) for all \(t\geq0\). Then, for every
\(x\in M\) such that
\[
        \rho_{\partial}(x)\leq r_0\wedge J^{-1}(0),
\]
the Laplacian comparison estimate gives
\[
        \Delta\rho_{\partial}(x)
        \geq
        (d-1)\frac{J'}{J}(\rho_{\partial}(x)).
\]
When \(k>0\), the first zero is given by
\[
        J^{-1}(0)
        =
        \frac1{\sqrt{k}}
        \arcsin\left(\sqrt{\frac{k}{k+\theta^2}}\right).
\]

Using this assumption, F.-Y. Wang constructed a conformal factor \(\phi\)
satisfying
\[
        N\log\phi\geq \sigma
        \quad \hbox{on } \partial M;
\]
see \cite{Wang:2007} or \cite[Theorem~3.2.9]{Wbook14}. Following this
construction, set
\[
        r_1:=r_0\wedge J^{-1}(0)
\]
and define
\[
\begin{aligned}
\log \phi(x)
:=
\frac{\sigma}{\alpha}
\int_{0}^{\rho_{\partial}(x)\wedge r_1}
[J(s)-J(r_1)]^{1-d}
\left(
\int_{s}^{r_1}
[J(u)-J(r_1)]^{d-1}\,du
\right)ds,
\end{aligned}
\]
where
\[
        \alpha:=
        (1-J(r_1))^{1-d}
        \int_0^{r_1}
        [J(s)-J(r_1)]^{d-1}\,ds .
\]
Then, by the proof of \cite[Theorem~1.1]{wang2005functional},
\((L\log\phi)^-\) and \(|\nabla\log\phi|^2\) are bounded. In particular,
this \(\phi\) satisfies the assumptions required in
Theorem~\ref{s1.thm.3}. Moreover,
\[
        1\leq \phi\leq \e^{\frac12 d\sigma r_1}
\]
and
\[
        -L\log\phi+2|\nabla\log\phi|^2\leq C_0,
\]
where
\begin{align}\label{def-C0}
        C_0=
        \sigma\left(\delta_{r_1}(Z)+\frac{d}{r_1}\right)+2\sigma^2
\end{align}
and
\[
        \delta_{r_1}(Z)
        :=
        \sup\{|Z(x)|:\rho_{\partial}(x)\leq r_1\}.
\]

\begin{corollary}\label{cor-S3-1}
Assume Conditions {\bf(C2)} and {\bf(C3)}. Suppose that \(h\in C^1(M)\) is
non-negative and that there exist non-negative functions
\(\alpha_0,\alpha_1\) on \(M\) such that, for every \(x\in M\),
\begin{align}\label{gradient-ineq-2}
        \sup_{s\in[0,1]}
        \mathbb E^x[h^2(x_s)]
        \leq
        \alpha_0(x),
\ \text{ and }\ 
        \sup_{s\in[0,1]}
        \mathbb E^x[|\nabla h|^2(x_s)]
        \leq
        \alpha_1(x).
\end{align}
Then, for every \(T>0\), \(x\in M\), and \(f\in\mathcal B_b(M)\),
\[
\begin{aligned}
|\nabla P_Tf|(x)
\leq
\e^{\frac12 d\sigma r_1}
\biggl[
&
\frac{1}{\sqrt{\min\{T,1\}}}
+
\sqrt{\alpha_0(x)}+C_0
+\frac{1}{4}\sqrt{\alpha_1(x)}+
\frac18
\bigl(\sqrt{\alpha_0(x)}+C_0\bigr)\sqrt{\alpha_1(x)}
\biggr]
\|f\|_{\infty}.
\end{aligned}
\]
\end{corollary}

\begin{proof}
Set
\[
        t_0:=T\wedge1.
\]
By the semigroup property,
\[
        P_Tf=P_{t_0}P_{T-t_0}f.
\]
Applying Theorem~\ref{s1.thm.3} to \(P_{T-t_0}f\) with time parameter
\(t_0\), and using
\[
        \|P_{T-t_0}f\|_\infty\leq \|f\|_\infty,
\]
we obtain the desired estimate once the coefficients in
Theorem~\ref{s1.thm.3} are controlled.

By the construction of \(\phi\) under Condition {\bf(C3)},
\[
        \frac{\|\phi\|_\infty}{\phi(x)}
        \leq
        \e^{\frac12d\sigma r_1}.
\]
Moreover,
\[
        -L\log\phi+2|\nabla\log\phi|^2\leq C_0.
\]
Thus we may take \(\psi=C_0\) in Theorem~\ref{s1.thm.3}. Since \(h\geq0\),
\[
        |h+\psi|\leq h+C_0.
\]
Hence, for \(s\in[0,t_0]\subset[0,1]\),
\[
\begin{aligned}
\left(\mathbb E^x[(h(x_s)+C_0)^2]\right)^{1/2}
&\leq
\left(\mathbb E^x[h^2(x_s)]\right)^{1/2}+C_0 \\
&\leq
\sqrt{\alpha_0(x)}+C_0 .
\end{aligned}
\]
Also, since \(\psi=C_0\) is constant,
\[
        \nabla(h+\psi)=\nabla h,
\]
and therefore
\[
        \sup_{s\in[0,t_0]}
        \mathbb E^x[|\nabla(h+\psi)|^2(x_s)]
        \leq
        \alpha_1(x).
\]
Substituting these bounds into Theorem~\ref{s1.thm.3}, and using
\(t_0\leq1\), gives
\[
\begin{aligned}
|\nabla P_Tf|(x)
\leq
\e^{\frac12d\sigma r_1}
\biggl[
&
\frac{1}{\sqrt{t_0}}
+
\frac14\sqrt{\alpha_1(x)}\,t_0
+
(\sqrt{\alpha_0(x)}+C_0)\sqrt{t_0}
\\
&\quad
+
\frac{t_0^2}{8}
(\sqrt{\alpha_0(x)}+C_0)\sqrt{\alpha_1(x)}
\biggr]
\|f\|_{\infty}.
\end{aligned}
\]
Since \(0<t_0\leq1\), this implies
\[
\begin{aligned}
|\nabla P_Tf|(x)
\leq
\e^{\frac12d\sigma r_1}
\biggl[
&
\frac{1}{\sqrt{\min\{T,1\}}}
+
\sqrt{\alpha_0(x)}+C_0
+\frac{1}{4}\sqrt{\alpha_1(x)}+
\frac18
(\sqrt{\alpha_0(x)}+C_0)\sqrt{\alpha_1(x)}
\biggr]
\|f\|_{\infty}.
\end{aligned}
\]
The proof is complete.
\end{proof}

\begin{remark}
If the lower bound for the Bakry--Emery curvature has the same form as in
Corollary~\ref{examples}, then the coefficients \(\alpha_0\) and
\(\alpha_1\) in \eqref{gradient-ineq-2} can be made explicit in the same way
as in Corollary~\ref{examples}.
\end{remark}

\section{Application to Neumann eigenfunction estimates}
We now turn to a direct application of the preceding semigroup gradient
estimates to pointwise bounds for Neumann eigenfunctions.  For compact manifolds, related estimates were studied in \cite{ArnaudonThalmaierWang2020, ChengThalmaierWang2024Hessian}. 
 Let \(u\in C^2(M)\) be a classical
Neumann eigenfunction of \(L\), namely
\begin{align}\label{Harmonic-function}
        Lu=-\lambda u \quad \hbox{in } M^\circ,
        \qquad
        Nu|_{\partial M}=0,
\end{align}
where \(\lambda>0\). Equivalently, \(u\) belongs to the domain of the
Neumann realization of \(L\) and satisfies the above equation.
Since the Neumann semigroup \(P_t\) is generated by \(\frac12L\), the spectral relation gives
\[
        P_tu=\e^{-\lambda t/2}u,\qquad t\geq0.
\]
Consequently,
\[
        |\nabla u|(x)=\e^{\lambda T/2}|\nabla P_Tu|(x),
        \qquad T>0.
\]
Thus the gradient estimates for \(P_T\) immediately yield gradient estimates
for Neumann eigenfunctions.  

We first treat the convex boundary case, where the boundary contribution is
non-positive and no conformal change is needed.  

\begin{corollary}\label{cor:neumann-eigen-convex}
Assume Condition {\bf(C1)}. Suppose that there exist non-negative
functions \({\theta}_0,\, {\theta}_1\) on \(M\) such that, for every \(x\in M\),
\[
        \sup_{s\in [0,1] }\mathbb E^x[h^2(x_s)]\leq  {\theta}_0(x),
        \qquad
        \sup_{s\in [0, 1] }\mathbb E^x[|\nabla h|^2(x_s)]
        \leq  {\theta}_1(x).
\]
Let \(u\in C^2(M)\) satisfy \eqref{Harmonic-function}
with \(\lambda>0\). Then, for every \(x\in M\),
\[
        |\nabla u|(x) \leq
        \e^{1/2}
        \left[
        \sqrt{\max\{\lambda,  1\}}
        +
        \frac{\sqrt{{\theta}_1(x)}}{4}
        +
        \sqrt{{\theta}_0(x)}
        +
        \frac{\sqrt{{\theta}_0(x) {\theta}_1(x)}}{8}
        \right]
        \|u\|_{\infty}.
\]
\end{corollary}

\begin{proof}
For every \(T>0\), the spectral relation gives
\[
        |\nabla u|(x)=\e^{\lambda T/2}|\nabla P_Tu|(x).
\]
Applying Theorem~\ref{s1.thm.2} to \(f=u\), and using the
assumptions
\[
        \sup_{s\in [0,1] }\mathbb E^x[h^2(x_s)]
        \leq  {\theta}_0(x),
\]
and
\[
        \sup_{s\in[0,1] }\mathbb E^x[|\nabla h|^2(x_s)]
        \leq  {\theta}_1(x),
\]
we obtain
\[
        |\nabla u|(x)
        \leq
        \e^{\lambda T/2}
        \left[
        \frac1{\sqrt {\min\{T,\, 1\}}}
        +
        \frac14\sqrt{ {\theta}_1(x)}
        +
        \sqrt{ {\theta}_0(x)}
        +
        \frac{1}{8}
        \sqrt{ {\theta}_0(x) {\theta}_1(x)}
        \right]
        \|u\|_{\infty}.
\]
Taking \(T=\lambda^{-1}\) gives the desired estimate.
\end{proof}

We then consider the general
non-convex case by using a conformal factor \(\phi\) which compensates for the
negative part of the second fundamental form.

\begin{corollary}
\label{cor:neumann-eigen-nonconvex}
Assume Condition {\bf(C2)}. Suppose that there exists
\(\phi\in \mathcal{D}(M)\) such that
\[
        N\log\phi\geq \sigma
        \quad \hbox{on } \partial M,
\]
and that for some \(\psi\in C^1(M)\cap \mathcal K(M)\),
\[
        -L\log\phi+2|\nabla\log\phi|^2\leq \psi.
\]
Assume further that there exist non-negative functions \( \alpha_0,\, \alpha_1\)
on \(M\) such that, for every \(x\in M\),
\[
        \sup_{s\in [0,1]}\mathbb E^x[|h+\psi|^2(x_s)]
        \leq  {\alpha}_0(x),
        \qquad
        \sup_{s\in [0,1] }
        \mathbb E^x[|\nabla(h+\psi)|^2(x_s)]
        \leq  {\alpha}_1(x).
\]
Let \(u\in C^2(M)\) satisfy \eqref{Harmonic-function} with \(\lambda>0\). Then for every \(x\in M\),
\[
        |\nabla u|(x)
        \leq
        \e^{1/2}
        \frac{\|\phi\|_{\infty}}{\phi(x)}
        \left[
        \sqrt{\max\{\lambda,  1\} }
        +
        \frac{\sqrt{ {\alpha}_1(x)}}{4}
        +
       \sqrt{ {\alpha}_0(x)}
        +
        \frac{\sqrt{ {\alpha}_0(x) {\alpha}_1(x)}}{8}
        \right]
        \|u\|_{\infty}.
\]
\end{corollary}

\begin{proof}
As above,
\[
        |\nabla u|(x)=\e^{\lambda T/2}|\nabla P_Tu|(x),
        \qquad T>0.
\]
Applying Theorem~\ref{s1.thm.3} to \(f=u\), and using
the  bounds
\[
        \sup_{s\in[0,1]}\mathbb E^x[|h+\psi|^2(x_s)]
        \leq  {\alpha}_0(x),
        \qquad
        \sup_{s\in[0,1]}
        \mathbb E^x[|\nabla(h+\psi)|^2(x_s)]
        \leq {\alpha}_1(x),
\]
we get
\[
        |\nabla u|(x)
        \leq
        \e^{\lambda T/2}
        \frac{\|\phi\|_{\infty}}{\phi(x)}
        \left[
        \frac1{\sqrt {\min\{T,  1\}}}
        +
        \frac14\sqrt{ {\alpha}_1(x)}
        +
        \sqrt{ {\alpha}_0(x)}
        +
        \frac{1}{8}
        \sqrt{ {\alpha}_0(x) {\alpha}_1(x)}
        \right]
        \|u\|_{\infty}.
\]
Choosing \(T=\lambda^{-1}\) proves the claim.
\end{proof}

Finally, under the explicit
geometric assumptions in Condition {\bf(C3)}, the conformal factor can be
chosen explicitly, leading to a more concrete estimate.

\begin{corollary}
\label{cor:neumann-eigen-explicit-nonconvex}
Assume Conditions {\bf (C2)} and {\bf (C3)}. 
Suppose that there exist non-negative functions \( \alpha_0,\, \alpha_1\) on \(M\)
such that
\[
        \sup_{s\in [0,1]}\mathbb E^x[h^2(x_s)]\leq  {\alpha}_0(x),
        \qquad
        \sup_{s\in [0,1] }\mathbb E^x[|\nabla h|^2(x_s)]
        \leq  {\alpha}_1(x).
\]
Let \(u\in C^2(M)\) satisfy \eqref{Harmonic-function} with \(\lambda>0\). Then, for every \(x\in M\),
\[
        |\nabla u|(x)
        \leq
        \e^{\frac12(1+d\sigma r_1)}
        \left[
        \sqrt{\max\{\lambda,  1\}}
        +
        \frac{\sqrt{ {\alpha}_1(x)}}{4}
        +
        \bigl(\sqrt{ {\alpha}_0(x)}+C_0\bigr)
       \Bigl(1+ \frac{\sqrt{{\alpha}_1(x)}}{8}\Bigr) \right]
        \|u\|_{\infty},
\]
where $C_0$ is defined in \eqref{def-C0}. 
\end{corollary}

\begin{remark}
The preceding corollaries show that the leading high-frequency contribution is
of order
\(
        \sqrt{\lambda}\,\|u\|_{\infty},
\)
which is consistent with \cite{ArnaudonThalmaierWang2020} when the domain is compact
and the Ricci curvature is lower bounded by a constant.  The remaining terms do not affect the leading growth
 once the curvature and
boundary quantities are well controlled. 
\end{remark}

\section{Appendix}

We point out that pointwise moment assumptions on the curvature lower bound
do not by themselves imply the Kato condition, nor do they guarantee the
exponential integrability needed for the multiplicative functional.

\begin{example}\label{no-Kato}
Let \(M\) be a rotationally symmetric model manifold
\[
        M=[0,\infty)\times \mathbb S^{d-1},
        \qquad
        g=dr^2+\varphi(r)^2 g_{\mathbb S^{d-1}},
\]
where \(\varphi(r)\sim r\) near \(r=0\), and
\[
        \varphi(r)=\exp(ar^2)
\]
for all sufficiently large \(r\), with \(a>0\). We take \(Z=0\), so that
\(\operatorname{Ric}_Z=\operatorname{Ric}\). 
On a model manifold, the radial
Ricci curvature is given by
\[
        \operatorname{Ric}(\partial_r,\partial_r)
        =
        -(d-1)\frac{\varphi''(r)}{\varphi(r)}.
\]
Since, for large \(r\),
\[
        \frac{\varphi'(r)}{\varphi(r)}=2ar,
        \qquad
        \frac{\varphi''(r)}{\varphi(r)}
        =
        2a+4a^2r^2,
\]
we have
\[
        \operatorname{Ric}(\partial_r,\partial_r)
        =
        -(d-1)(2a+4a^2r^2)
\]
for large \(r\). Hence the negative part of the Ricci curvature can be
controlled by
\[
        h(x)=C_a(1+\rho_o(x)^2),
\]
where \(C_a>0\) is a constant of order \(a^2\). In particular,
\[
        \operatorname{Ric}\geq -h.
\]
\end{example}

Let \(x_s\) be the Brownian motion on \(M\), generated by
\(\frac12\Delta\), and set
\[
        R_s=\rho_o(x_s).
\]
The radial process satisfies, away from the cut locus,
\[
        dR_s=d\beta_s+
        \frac{d-1}{2}\frac{\varphi'(R_s)}{\varphi(R_s)}\,ds,
\]
where \(\beta_s\) is a one-dimensional Brownian motion. Since
\(\varphi'(r)/\varphi(r)=2ar\) for large \(r\), the radial drift has at most
linear growth. Consequently, for every \(q\geq1\) and every finite \(t>0\),
\[
        \sup_{s\in[0,t]}\mathbb E^x R_s^q<\infty
\]
for each fixed starting point \(x\in M\). Therefore
\[
        \sup_{s\in[0,t]}\mathbb E^x h(x_s)^2<\infty,
        \qquad
        \sup_{s\in[0,t]}\mathbb E^x |\nabla h|^2(x_s)<\infty.
\]
Indeed,
\[
        h(x_s)^2
        =
        C_a^2(1+R_s^2)^2
        \leq C(1+R_s^4),
\]
and
\[
        |\nabla h|^2(x_s)
        =
        4C_a^2 R_s^2
\]
outside a compact set, up to a harmless modification inside the compact part.

However, these fixed-point moment bounds do not imply the exponential
integrability of the path integral. We sketch the reason. Choose an initial
point \(x\) with \(R_0=\rho_o(x)>2\). Since the radial drift is non-negative
for large \(r\), on suitable Brownian events one has a lower bound of the form
\[
        R_s\geq n,
        \qquad s\in[t/2,t].
\]
More precisely, for large integers \(n\), consider
\[
A_n=
\left\{
\beta_{t/2}\in[n,n+1],
\quad
\inf_{0\leq s\leq t/2}\beta_s\geq -1,
\quad
\inf_{t/2\leq s\leq t}(\beta_s-\beta_{t/2})\geq -1
\right\}.
\]
By the Markov property and the reflection principle for one-dimensional
Brownian motion, there exist constants \(c_1,c_2>0\) such that
\[
        \mathbb P(A_n)\geq c_1 \e^{-c_2 n^2}.
\]
On \(A_n\), we have \(R_s\geq n\) for \(s\in[t/2,t]\). Hence
\[
        \int_0^t h(x_s)\,ds
        \geq
        \frac t2 C_a n^2.
\]
Therefore
\[
\begin{aligned}
\mathbb E^x
\exp\left(\int_0^t h(x_s)\,ds\right)
&\geq
\sum_{n\gg1}
\mathbb E^x
\left[
\exp\left(\int_0^t h(x_s)\,ds\right)\mathbf 1_{A_n}
\right]  \\
&\geq
c_1
\sum_{n\gg1}
\exp\left[
\left(\frac t2 C_a-c_2\right)n^2
\right].
\end{aligned}
\]
Since \(C_a\) is of order \(a^2\), choosing \(a>0\) sufficiently large gives
\[
        \frac t2 C_a>c_2,
\]
and hence
\[
        \mathbb E^x
        \exp\left(\int_0^t h(x_s)\,ds\right)
        =
        \infty .
\]

We finally verify directly from the definition that this function \(h\) is not
in the Kato class. Recall that
\[
        f\in\mathcal K(M)
        \quad\Longleftrightarrow\quad
        \lim_{\alpha\downarrow0}
        \sup_{x\in M}
        \int_0^\alpha
        \mathbb E^x |f(x_s)|\,ds
        =
        0.
\]
For \(h(x)=C_a(1+\rho_o(x)^2)\), this condition fails. Indeed, for the squared
distance function one has, outside the cut locus and hence in the barrier
sense,
\[
        \frac12\Delta \rho_o^2
        =
        1+(d-1)\rho_o\frac{\varphi'(\rho_o)}{\varphi(\rho_o)}.
\]
Since \(\varphi'(r)/\varphi(r)=2ar\) for large \(r\), and the remaining part of
the manifold is compact, there exists a constant \(C>0\) such that globally
\[
        \frac12\Delta \rho_o^2\geq -C.
\]
By It\^o's formula,
\[
        \mathbb E^x \l[ \rho_o^2(x_s)\r]
        =
        \rho_o^2(x)+
       \frac12 \int_0^s
       \mathbb E^x\l[ \Delta  \rho_o^2(x_r)\r] \,dr
        \geq
        \rho_o^2(x)-Cs.
\]
Now choose \(X_R\in M\) such that \(\rho_o(X_R)=R\). For every fixed
\(\alpha>0\), if \(R\) is sufficiently large, then
\[
        \mathbb E^{X_R}\rho_o^2(x_s)
        \geq
        \frac12 R^2,
        \qquad 0\leq s\leq \alpha.
\]
Consequently,
\[
\begin{aligned}
\int_0^\alpha
\mathbb E^{X_R} h(x_s)\,ds
&=
C_a\int_0^\alpha
\mathbb E^{X_R}
\left(1+\rho_o(x_s)^2\right)\,ds  \geq
\frac12 C_a\alpha R^2.
\end{aligned}
\]
Letting \(R\to\infty\), we obtain
\[
        \sup_{x\in M}
        \int_0^\alpha
        \mathbb E^x h(x_s)\,ds
        =
        \infty
\]
for every \(\alpha>0\). Hence
\[
        h\notin\mathcal K(M).
\]

\providecommand{\bysame}{\leavevmode\hbox to3em{\hrulefill}\thinspace}
\providecommand{\MR}{\relax\ifhmode\unskip\space\fi MR }
\providecommand{\MRhref}[2]{%
  \href{http://www.ams.org/mathscinet-getitem?mr=#1}{#2}
}
\providecommand{\href}[2]{#2}

\end{document}